\documentclass[11pt,reqno]{amsart}
\usepackage{hyperref}
\usepackage{graphicx}
\usepackage{color}
\usepackage[all]{xy}
\usepackage{amsfonts,amssymb}
\usepackage{amsthm}

\usepackage{mathrsfs}

\usepackage{a4wide}

\newtheorem{neu}{}[section]
\newtheorem{Cor}[neu]{Corollary}
\newtheorem*{Cor*}{Corollary}
\newtheorem{Thm}[neu]{Theorem}
\newtheorem*{Thm*}{Theorem}

\newtheorem{Prop}[neu]{Proposition}
\newtheorem*{Prop*}{Proposition}
\theoremstyle{definition}
\newtheorem{Lemma}[neu]{Lemma}

\newtheorem*{Rmk*}{Remark}
\newtheorem{Rmk}[neu]{Remark}

\newtheorem*{Ex*}{Example}

\newtheorem*{Qu*}{Question}

\newtheorem{Def}[neu]{Definition}

\newcommand{\N}{\mathbb{N}}

\newcommand{\Z}{\mathbb{Z}}
\newcommand{\R}{\mathbb{R}}

\newcommand{\CZ}{\mu_{\mathrm{CZ}}}

\newcommand{\im}{\mathrm{im\,}}
\newcommand{\id}{\mathrm{id}}

\newcommand{\om}{\omega}
\newcommand{\Om}{\Omega}
\newcommand{\ev}{\mathrm{ev}}

\newcommand{\LW}{\mathrm{LW}}
\newcommand{\Rb}{\mathcal{R}}

\newcommand{\nbd}{neighborhood }
\newcommand{\rank}{\mathrm{rank}}

\newcommand{\A}{\mathcal{A}}

\renewcommand{\S}{\mathcal{S}}

\newcommand{\F}{\mathcal{F}}

\newcommand{\D}{\mathcal{D}}

\newcommand{\M}{\mathcal{M}}

\newcommand{\J}{\mathcal{J}}

\newcommand{\E}{\mathcal{E}}
\renewcommand{\L}{\mathcal{L}}

\renewcommand{\H}{\mathrm{H}}

\newcommand{\Ham}{\mathrm{Ham}}

\newcommand{\FC}{\mathrm{FC}}

\newcommand{\FH}{\mathrm{FH}}

\newcommand{\Crit}{\mathrm{Crit}}

\newcommand{\V}{\mathcal{V}}

\newcommand{\Hc}{\mathcal{H}}
\newcommand{\Bc}{\mathcal{B}}
\newcommand{\Nc}{\mathcal{N}}
\newcommand{\beq}{\begin{equation}}
\newcommand{\beqn}{\begin{equation}\nonumber}
\newcommand{\eeq}{\end{equation}}
\newcommand{\bea}{\begin{equation}\begin{aligned}}
\newcommand{\bean}{\begin{equation}\begin{aligned}\nonumber}
\newcommand{\eea}{\end{aligned}\end{equation}}

\numberwithin{equation}{section}

\definecolor{Urs}{rgb}{0,.7,0}
\definecolor{Youngjin}{rgb}{0,0,1}
\definecolor{red}{rgb}{1,0,0}

\newcommand{\He}{\mathscr{H}}
\newcommand{\p}{\partial}
\newcommand{\Mf}{\mathfrak{M}}

\newcommand{\g}{{\mathfrak g}}         

\newcommand{\intinf}{{\int_{-\infty}^{\infty}}}

\begin{document}
\title[On magnetic leaf-wise intersections]
{On magnetic leaf-wise intersections}
\author{Youngjin Bae}
\address{
    Youngjin Bae\\
    Department of Mathematics and Research Institute of Mathematics\\
    Seoul National University}
\email{jini0919@snu.ac.kr}

\begin{abstract}
In this paper we study leaf-wise intersections 
in a setting where the Hamiltonian perturbation has magnetic effects. 
In particular, we establish their existence 
under certain topological assumptions on the ambient manifold.
\end{abstract}
\maketitle

\section{Introduction}
Let $(N,g)$ be a closed connected orientable Riemannian manifold 
and $\tau:T^*N\to N$ be its cotangent bundle with the canonical symplectic form 
$\om_{std}=dp\wedge dq$. 
Here $(q,p)$ are the canonical coordinates on $T^*N$.
Let $\Sigma$ be a hypersurface in an exact symplectic manifold $(T^*N,d\lambda,\lambda=p\,dq)$
such that $(\Sigma,\alpha:=\lambda|_{\Sigma})$ is a contact manifold.

As a main example, such a hypersurface can be obtained from 
a sufficiently high energy level of a mechanical Hamiltonian
\bean
G=\frac{1}{2}|p|^2+U(q),
\eea
i.e. a closed energy hypersurface $G^{-1}(k)\subset T^*N,\ k>\max_{q\in N}U(q)$
is a contact manifold with a contact form $\lambda|_{G^{-1}(k)}$.
Here $|p|$ denotes the dual norm of the Riemannian metric $g$ on $N$
and $U:N\to\R$ is a smooth potential.
This Hamiltonian system describes the motion of a particle on $N$ 
subject to the conservative force $-\nabla U$.

The contact hypersurface $\Sigma$ is foliated by the leaves of the characteristic line bundle 
which is spanned by the Reeb vector field $R$ of $\alpha$.
Let $\phi_t^\Sigma:\Sigma\to\Sigma$ be the flow of $R$.
For $x\in\Sigma$ we denote by $L_x$ the leaf through $x$ which can be parametrized as
$L_x=\{\phi_t^\Sigma(x)\,:\,t\in\R\}$.
If $L_x$ is closed, we call $L_x$ {\em a closed Reeb orbit}
and $x$ {\em a periodic point}.

In order to introduce leaf-wise intersections,
we need the following definitions.
\begin{Def}\label{def:betadef}
A {\em magnetic perturbation} $\sigma\in C^\infty(\R/\Z,\Om^2(N))$ 
is a time-dependent closed 2-form on $N$ which satisfies the following conditions:
\begin{itemize}
\item $\widetilde\sigma(t)=d\theta(t)$, for some $\theta\in C^\infty(\R/\Z,\Om^1(\widetilde N))$ for all $t\in\R/\Z$;
\item $\sigma(t)=0$, $\theta(t)=0$ for all $t\in[0,\frac{1}{2}]$;
\item $\theta(t)\in\Om^1(\widetilde N)$ is bounded for all $t\in[\frac{1}{2},1]$.
\end{itemize}
Here $\widetilde\sigma(t)$ is the lift of $\sigma(t)$ to the universal cover $\widetilde N$ and
such a 2-form $\sigma(t)$ which has a bounded primitive on the universal cover is called {\em $\widetilde d$-bounded}.
\end{Def}
Let $\Mf$ be the set of such magnetic perturbations and $\mathcal P$ be the set of primitives of magnetic perturbations on the universal cover.
We consider an $\R/\Z$-parametrized symplectic form $\om_\sigma:\R/\Z\to\Om^2(T^*N)$ as follows
\bea\label{eqn:om_m}
\om_\sigma:=\om_{\rm std}+\tau^*\sigma.
\eea

\begin{Def}\label{def:perham}
We define
\bean
&\Hc:=\{H\in C^\infty_c(\R/\Z\times T^*N)\,:\,H(t,\cdot)=0,\ \forall t\in[0,\frac{1}{2}]\}.
\eea
Throughout this article, we fix an $H\in\Hc$ and
the diffeomorphism $\varphi$ which is the time-1-map of
the Hamiltonian vector field $X_H^\sigma$ defined by $\iota_{X_H^\sigma}\om_\sigma=-dH$.
When $\sigma\equiv 0$, the diffeomorphism $\varphi$ becomes 
a (non-magnetic) Hamiltonian diffeomorphism of $(T^*N,\om_{std})$.
\end{Def}

\begin{Def}\label{def:m_lwip}
A point $x\in\Sigma$ is called a 
{\em magnetic leaf-wise intersection point},
if $\varphi(x)\in L_x$.
In other words, there exists $\eta\in\R$ such that
\bea\label{eqn:lwip}
\phi_\eta^\Sigma(\varphi(x))=x.
\eea
\end{Def}

Note that a leaf-wise intersection point is 
the $\sigma=0$ case of a magnetic leaf-wise intersection point.
The leaf-wise intersection problem asks whether a given diffeomorphism has
a leaf-wise intersection point in a given hypersurface $\Sigma$.
If there exist leaf-wise intersections 
one can ask further a lower bound on the number of leaf-wise intersections.
This problem was introduced by Moser in \cite{Mos}, 
and studied further in \cite{Ban, EH, Hof, Gin, Dra, Gur, Zil, 
AF10b, AF10c, AM09, AM10, Kan09, Kan10, Kan10b, Mer10}.
See \cite{AF09} for the brief history of these problems.
In this article, we investigate the approaches in \cite{AF09b,MMP}
and generalize their results.

We call a hypersurface $\Sigma\subset T^*N$ {\em non-degenerate} 
if closed Reeb orbits on $\Sigma$ form a discrete set.
A generic $\Sigma$ is non-degenerate, see \cite[Theorem B.1]{CF09}.
If $\Sigma$ is non-degenerate, 
then periodic leaf-wise intersection points
can be excluded by choosing a generic Hamiltonian function, 
see \cite[Theorem 3.3]{AF09b}.
In this setting, a leaf-wise intersection point $x$
has the unique real number $\eta$ such that (\ref{eqn:lwip}) is satisfied.
By these reason, we only consider {\em non}-periodic 
(magnetic) leaf-wise intersection points.

In order to state the main result, 
we need the following notion.
Let $\L_N$ be the free loop space of $(N,g)$. 
The energy functional $\E:\L_N\to\R$ is given by
\bean
\E(q):=\int_0^1\frac{1}{2}|\dot q|^2dt.
\eea
For given $0<T<\infty$, denote by
\bean
\L_N(T):=\bigg\{q\in\L_N\,:\,\E(q)\leq\frac{1}{2}T^2\bigg\}.
\eea

Let $\Sigma$ be a non-degenerate fiberwise starshaped hypersurface and
$\varphi$ be a generic diffeomorphism.
Given $T>0$ let us define
\bean
n(T)=n_{\Sigma,\varphi}(T)&:=\#\{x\in T^*N\,:\,\phi_\eta^\Sigma(\varphi(x))=x,\ \text{for some }\eta\in(0,T)\}.
\eea

\begin{Thm}\label{thm:fconti}
Let $N$ be a closed connected oriented manifold of dimension $n\geq2$.
Let $\Sigma$ be a non-degenerate fiberwise starshaped hypersurface in $T^*N$.
Let $g$ be a bumpy Riemannian metric on $N$ with $S^*_g N$ contained
in the interior of the compact region bounded by $\Sigma$.
Assume that $\varphi$ is generic.
Then there exists a constant $c=c(N,g,\Sigma,\varphi)>0$
such that the following holds:
For all sufficiently large $T>0$,
\bea\label{eqn:mlwipesloophm}
n(T)\geq \frac{1}{c}\cdot\rank\{\iota:\H_*\big(\L_N(c(T-1))\big)\to\H_*(\L_N)\}.
\eea
\end{Thm}

Under a certain topological assumption on $N$, 
the right hand side of (\ref{eqn:mlwipesloophm})
grows exponentially with $T$.
We denote by $\widetilde\pi_1(N)$ the set of conjugacy classes of $\pi_1(N)$.
Then the connected components of $\L_N$ corresponds to the elements of $\widetilde\pi_1(N)$,
hence the exponential growth rate of $\widetilde\pi_1(N)$ implies that
\bea\label{eqn:rankexp}
\liminf_{T\to\infty}\rank\{\iota:\H_0(\L_N(T))\to\H_0(\L_N)\}
\eea
has also exponential growth with respect to $T$, see \cite{Mc}.
Then the following corollary comes from 
exponential growth of (\ref{eqn:rankexp}) and Theorem \ref{thm:fconti}.

\begin{Cor}
Let $N$ be a closed connected oriented manifold of dimension $n\geq2$.
Let $\Sigma$ be a non-degenerate fiberwise starshaped hypersurface in $T^*N$.
Suppose that $\widetilde\pi_1(N)$ has exponential growth.
If $\varphi$ 
is generic then $n(T)$ grows exponentially with $T$.
\end{Cor}

The main example of such $N$ is any surface of genus greater than one.
In these cases, the magnetic field $\sigma$ can be chosen by the volume form of that surface.
Other candidates for $N$ are the {\em symplectically hyperbolic} manifolds
which will be discussed in Definition \ref{def:symphyp}, Proposition \ref{prop:symphypexp}.

In proving Theorem \ref{thm:fconti}, 
we heavily need a recent result by Marcarini, Merry and Paternain \cite{MMP}.
With the same assumption as in Theorem \ref{thm:fconti}
and a generic non-magnetic Hamiltonian, they showed
the exponential growth rate of leaf-wise intersections.
In this paper, we extend their result to the magnetic case by using a method developed in \cite{BF}.

More precisely we will use a variational approach associated to magnetic leaf-wise intersections.
The main tools are Rabinowitz Floer homology and its variations.
In order to show the main result, invariance of Rabinowitz Floer homology is crucial and
we take advantage of continuation map between two different Rabinowitz Floer chain complexes.

\vspace{3mm}

\emph{Acknowledgement: }  
I am grateful to my advisor Urs Frauenfelder for fruitful discussions
and detailed comments.
I also thank the anonymous referee.
The author is supported by the Basic
research fund 2010-0007669 funded by the Korean government.

\section{A perturbation of the Rabinowitz action functional}\label{sec:pert}

Let us begin with a {\em defining Hamiltonian} $\bar{F}$ of a contact hypersurface $\Sigma\subset T^*N$ 
which is constant outside of a compact set containing $\Sigma$.
\begin{Def}
Given a fiberwise starshaped hypersurface $\Sigma\subset T^*N$,
\bean
\bar{\D}(\Sigma):=\{\bar{F}\in C^\infty(T^*N)\,:\,\bar{F}^{-1}(0)=\Sigma,\ X_{\bar{F}}|_\Sigma=R,\ X_{\bar{F}}\text{ is compactly supported }\}.
\eea
\end{Def}

With the defining Hamiltonian $\bar{F}\in\bar{\D}(\Sigma)$, 
the Rabinowitz action functional
$\A^{\bar{F}}:\L\times\R\to\R$ is defined by
\bean
\A^{\bar{F}}(u,\eta):=\int_0^1 u^*\lambda-\eta\int_0^1 \bar{F}(u(t))dt.
\eea
Here $\L=\L_{T^*N}:=C^\infty(\R/\Z,T^*N)$.
The critical points of $\A^{\bar{F}}$ satisfy
\bea\label{eqn:critA^F0}
\left.
\begin{array}{cc}
\frac{d}{dt}u(t)=\eta X_{\bar{F}}(u(t)) \\
\bar{F}(u(t))dt=0.
\end{array}
\right\}
\eea
Since the restriction of the Hamiltonian vector field $X_{\bar{F}}$
to $\Sigma$ is the Reeb vector field, the equations (\ref{eqn:critA^F0})
are equivalent to
\bean
\left.
\begin{array}{cc}
\frac{d}{dt}u(t)=\eta R(u(t)) \\
u(t)\in\Sigma,
\end{array}
\right\}
\eea
i.e. $u$ is a periodic orbit of the Reeb vector field on $\Sigma$ with period $\eta$.

If $\A^{\bar{F}}$ is Morse-Bott, then $\FH_*(\A^{\bar{F}})$ is well-defined, see \cite{CF09}.
By the work of Abbondandolo-Schwarz \cite{AS09} and
Cieliebak-Frauenfelder-Oancea \cite{CFO09},
we then have the following non-vanishing result when $*\neq0,1$.
\bea\label{eqn:RFHloop}
\FH_*(\A^{\bar{F}})=
\left\{
\begin{array}{ll}
\H_*(\L_N),&\text{if }\ *>1, \\
\H^{-*+1}(\L_N),&\text{if }\ *<0.
\end{array}
\right.
\eea
Here $\FH_*(\A^{\bar{F}})$ is the Floer homology for $\A^{\bar{F}}$
and $\L_N$ is the free loop space of $N$.

Now we introduce a {\em time-dependent defining Hamiltonian}
in order to consider an action functional whose critical points
give rise to leaf-wise intersection points.

\begin{Def}\label{def:defham}
Given a fiberwise starshaped hypersurface $\Sigma\subset T^*N$,
\bean
\D(\Sigma):=\{F\in C^\infty(\R/\Z\times T^*N)\,:\,F(t,x)=\rho(t)\bar{F}(x),\ \bar{F}\in\bar{\D}(\Sigma)\}.
\eea
Here $\rho:\R/\Z\to\R^{\geq0}$ satisfies
\bea\label{eqn:rhocondi}
\int_0^1\rho(t)dt=1 \quad\text{and}\quad \mathrm{supp}(\rho)\subset (0,\frac{1}{2}).
\eea
Remark that the Hamiltonian vector field $X_F$ satisfies
\bea\label{eqn:hamvec}
X_F(t,x)=\rho(t)X_{\bar{F}}(x).
\eea
\end{Def}

A leaf-wise intersection point $x\in\Sigma$ 
with respect to $\varphi_0\in\Ham_c(T^*N)$ 
can be interpreted as 
a critical point of a perturbed Rabinowitz action functional 
$\A^F_H:\L\times\R\to\R$ defined by
\bean
\A^F_H(u,\eta)=\int_0^1 u^*\lambda-\eta\int_0^1 F(t,u)dt-\int_0^1 H(t,u)dt.
\eea
Here the additional Hamiltonian $H:T^*N\to\R$ generates $\varphi_0$.
Then a critical point $(u,\eta)$ of $\A^F_H$ is a solution of
\bea\label{eqn:critFH}
\left.
\begin{array}{cc}
\frac{d}{dt}u(t)=\eta X_F(t,u(t))+X_H(t,u(t)) \\
\int_0^1F(t,u(t))dt=0.
\end{array}
\right\}
\eea
We observed in \cite{AF09} that if $(u,\eta)$ is a critical point of $\A^F_H$
then $u(\frac{1}{2})\in\Sigma$ is a leaf-wise intersection point. 
For a generic Hamiltonian $H$ for which $\A^F_H$ is Morse, 
Albers-Frauenfelder \cite{AF09b} constructed an isomorphism
\bea\label{eqn:AFHAF}
\FH(\A^F_H)\cong\FH(\A^{\bar{F}}).
\eea

Now we construct an action functional whose critical points give rise to
magnetic leaf-wise intersection points with respect to $\varphi=\phi_{X^\sigma_H}^1$. 
Throughout this article, we fix a primitive $\theta\in\mathcal P$ as in Definition \ref{def:betadef}.
We then define an additional term
\bea\label{eqn:beta_sigma_theta}
\Bc_\theta:\L&\to\R \\
u&\mapsto\int_0^1\widetilde\tau^*\theta_t(\widetilde u(t))[\frac{d}{dt}\widetilde u(t)]dt.
\eea
Here $\widetilde\tau:T^*\widetilde N\to\widetilde N$, 
$\widetilde u:\R/\Z\to T^*\widetilde N$
is a lifting of $u$.
We also fix a fundamental region $\underline N\subset\widetilde N$ and
assume that $\widetilde u(0)\in\underline N$.

A new action functional 
$\A_\theta:\L\times\R\to\R$ is defined by
\bean
\A_\theta(u,\eta)&=\A^F_{H,\theta}(u,\eta):=\A^F_H(u,\eta)+\Bc_\theta(u)  \\
&=\int_0^1u^*\lambda-\eta\int_0^1F(t,u(t))dt-\int_0^1H(t,u(t))dt
+\int_0^1\widetilde\tau^*\theta_t(\widetilde u(t))[\frac{d}{dt}\widetilde u(t)]dt.
\eea
A critical point $(u,\eta)\in\L\times\R$ of $\A_\theta$ satisfies
\bea\label{eqn:critm}
\left.
\begin{array}{cc}
\frac{d}{dt}u(t)=\eta X_F(t,u(t))+X_H^\sigma(t,u(t)) \\
\int_0^1F(t,u(t))dt=0.
\end{array}
\right\}
\eea
For convenience,
\bean
\Crit(\A_\theta)&:=\{w=(u,\eta)\in\L\times\R\,:\,(u,\eta)\text{ satisfies (\ref{eqn:critm})}\}.
\eea
In the following proposition we interpret the critical point 
as a magnetic leaf-wise intersection point as in Definition \ref{def:m_lwip}.

\begin{Prop}\label{prop:crit}
Let $(u,\eta)\in\Crit(\A_\theta)$. Then $x=u(\frac{1}{2})$ satisfies $\varphi(x)\in L_x$.
Thus $x$ is a magnetic leaf-wise intersection point.
\end{Prop}

\begin{proof}
For $t\in[0,\frac{1}{2}]$ we compute, using $H(t,\cdot)=0$ for all $t\leq\frac{1}{2}$,
\bean
\frac{d}{dt}\bar{F}(u(t))&=d\bar{F}(u(t))\cdot\frac{d}{dt}u(t) \\
&=d\bar{F}(u(t))\cdot[\eta \underbrace{X_F(t,u)}_{=\rho(t)X_{\bar{F}}(u)}+\underbrace{X_H^\sigma(t,u)}_{=0}]=0, \\
\eea
since $d\bar{F}(X_{\bar{F}})=0$. Hence $\bar{F}(u(t))=c$ for some constant $c$ when $t\leq\frac{1}{2}$. Thus,
\[
0=\int_0^1F(t,u)dt=\int_0^1\rho(t)\bar{F}(u(t))dt=c.
\]
Therefore $\bar{F}(u(t))=c=0$ and since $\bar{F}^{-1}(0)=\Sigma$, we have $u(t)\in\Sigma$ for $t\in[0,\frac{1}{2}]$.
In particular, $u(\frac{1}{2}),u(0)=u(1)\in\Sigma$.

For $t\in[\frac{1}{2},1]$ we have $F(t,\cdot)=0$. 
Thus the loop $u$ solves the equation $\frac{d}{dt}u(t)=X_H^\sigma(t,u)$ on $[\frac{1}{2},1]$,
and therefore, $u(1)=\varphi(u(\frac{1}{2}))$. We conclude that $\varphi(u(\frac{1}{2}))\in\Sigma$.
For $t\in[0,\frac{1}{2}]$, 
$\frac{d}{dt}u(t)=\eta X_F(t,u)+X_H^\sigma(t,u)=\eta X_F(t,u)=\eta R$,
since $X_F|_\Sigma=R$. 
This means that $\varphi(u(\frac{1}{2}))=u(1)=u(0)\in L_{u(\frac{1}{2})}$.
Thus $u(\frac{1}{2})$ is a magnetic leaf-wise intersection point.
\end{proof}

\subsection{Floer homology for $\A_\theta$}\label{sec:Flo(A_m)}
In this subsection, we show that $\FH(\A_\theta)$ is well-defined.
We assume that the readers are familiar with the construction in Floer theory
which can be found in \cite{Sal}.
Throughout this subsection, we follow the strategy in \cite{CF09} with minor modifications.

\begin{Def}\label{def:Mreg}
Let $\Sigma$ be a non-degenerate hypersurface in $T^*N$ with a defining Hamiltonian $F$.
A diffeomorphism $\varphi=\phi_{X^\sigma_H}^1$ 
or a pair $(H,\theta)\in\Hc\times\mathcal P$
is called {\em regular} with respect to $F\in\D(\Sigma)$ if
\begin{enumerate}
\item $\A_\theta=\A^F_{H,\theta}$ is Morse;
\item $\varphi$ has no periodic leaf-wise intersection points.
\end{enumerate}
For a given non-degenerate closed hypersurface $\Sigma$, $\varphi$ is regular
for generic $(H,\theta)\in\Hc\times\mathcal P$.
We discuss the regular property further in Appendix \ref{app:gen} and \ref{sec:noper}.
Throughout this article, we assume that $\varphi$ is regular.
\end{Def}

\begin{Rmk}
In order to define gradient flow lines, 
we need an $\R/\Z$-parametrized almost complex structure $J(t)$
which is compatible with the $\R/\Z$-parametrized symplectic form $\om_\sigma$.
This means that
\bean
g_t(\,\cdot\,,\,\cdot\,):=\om_\sigma(\,\cdot\,,J(t)\,\cdot\,)
\eea
defines a $\R/\Z$-parametrized inner product on $T^*N$.
We denote the set of such almost complex structures as $\J_\sigma$.

Given $J(t)\in\J_\sigma$, we denote by $\nabla_J\A_\theta$ the gradient of $\A_\theta$ with respect to the inner product
\bea\label{eqn:metricgJ}
\mathfrak g_J\big((\hat u_1,\hat\eta_1),(\hat u_2,\hat\eta_2)\big)
:=\int_0^1 g_t(\hat u_1,\hat u_2)dt+\hat\eta_1\hat\eta_2,
\eea
where $(\hat u_i,\hat\eta_i)\in T_{(u,\eta)}(\L\times\R)$ for $i=1,2$.
One can check that
\bean
\nabla_J\A_\theta(u,\eta)=
\left(
\begin{array}{cc}
-J(t,u)\big(\frac{d}{dt}u-X_H^\sigma(t,u)-\eta X_F(t,u)\big) \\
-\int_0^1F(t,u)dt
\end{array}
\right).
\eea
\end{Rmk}

\begin{Def}
A positive gradient flow line of $\A_{\theta}$ with respect to $J(t)\in\J_\sigma$ 
is a map $w=(u,\eta)\in C^\infty(\R,\L\times\R)$ solving the ODE
\bean
\frac{d}{ds}w(s)-\nabla_J\A_{\theta}(w(s))=0.
\eea
According to Floer's interpretation, this means that $u$ and $\eta$ are smooth maps
$u:\R\times (\R/\Z)\to T^*N$ and $\eta:\R\to\R$ satisfying
\bea\label{eqn:gradeqnm}
\left.
\begin{array}{cc}
\p_su+J(t,u)\big(\p_t u-X_H^\sigma(t,u)-\eta X_F(t,u)\big)=0 \\
\frac{d}{ds}\eta+\int_0^1F(t,u)dt=0.
\end{array}
\right\}
\eea
\end{Def}

\begin{Def}
The energy of a map $w\in C^\infty(\R,\L\times\R)$ is defined as
\bean
E(w):=\intinf\|\frac{d}{ds}w(s)\|^2_J ds,
\eea
where $\|\cdot\|_J=\sqrt{\mathfrak g_J(\cdot,\cdot)}.$
\end{Def}

Note that the energy of a given gradient flow line $w\in C^\infty(\R,\L\times\R)$ 
with limit conditions $\lim_{s\to\pm\infty}w(s)=w_\pm\in\Crit(\A_\theta)$ become
\bean
E(w)&=\intinf\|\frac{d}{ds}w(s)\|^2_J ds
=\intinf\mathfrak g_J(\nabla_J\A_\theta(w(s)),\frac{d}{ds}w(s))ds \\
&=\intinf\frac{d}{ds}\A_\theta(w(s))ds
=\A_\theta(w_+)-\A_\theta(w_-).
\eea

\begin{Rmk}\label{rmk:compact}
The most delicate issue in constructing $\FH(\A_\theta)$ 
is the compactness of the moduli space.
The moduli space of gradient flow lines $w(s)=(u,\eta)$ of $\A_\theta$
with the asymptotic conditions $\lim_{s\to\pm\infty}w(s)=w_{\pm}\in\Crit(\A_\theta)$
can be compactified up to
breaking of gradient flow lines, see Theorem \ref{thm:welldefAm}.
There are three analytic difficulties we have to overcome:
\begin{enumerate}
\item a uniform $L^\infty$ bound on $u\in\L$;
\item a uniform $L^\infty$ bound on $\eta\in\R$;
\item a uniform $L^\infty$ bound on the derivatives of $u\in\L$.
\end{enumerate}
Properties (1) and (3) are non-trivial but standard problems in Floer theory and
property (2) is also treated in \cite{CF09, AF09}.
However our gradient flow equation (\ref{eqn:gradeqnm}) is 
slightly generalized with the data $\sigma$.
Most of the remaining of this subsection 
is devoted to show property (2) in $\A_\theta$ case.
\end{Rmk}

\begin{Def}\label{def:c(m)}
Define a map $c:\Hc\times\mathcal P\to[0,\infty)$ by
\bean
c(H,\theta):=\sup_{(t,u)\in \R/\Z\times\L}
\left|\int_0^1(\widetilde\lambda+\widetilde\tau^*\theta_t)(\widetilde u(t))
[\widetilde X_H^\sigma(t,\widetilde u)]-H(t,u(t))dt\right|.
\eea
Note that $\widetilde\lambda+\widetilde\tau^*\theta$ 
is a primitive of $\widetilde\om_\sigma$ on the universal cover $T^*\widetilde N$.
We remind that $H$ is compactly supported,
and hence $c(H,\theta)$ is finite.
\end{Def}

\begin{Lemma}\label{lem:fund}
There exists $\epsilon>0$, $\widetilde c>0$ such that 
if $(u,\eta)\in \L\times\R$ satisfies
$\|\nabla_J\A_\theta(u,\eta)\|_J\leq\epsilon$ then
\beq\label{eqn:fund}
|\eta|\leq\widetilde c(|\A_\theta(u,\eta)|+1).
\eeq
Here $\|\cdot\|_J=\sqrt{\mathfrak g_J(\cdot,\cdot)}.$
\end{Lemma}

\begin{proof}
The proof consists of 3 steps.

\vspace{2mm}

{\bf Step 1} : {\em There exist $\delta>0$ and a constant $c_{\delta}<\infty$ 
such that if $u\in\L$ satisfies $(t,u(t))\in U_{\delta}=F^{-1}(-\delta,\delta)$ for all $t\in[0,\frac{1}{2}]$, 
then}
\bean
|\eta|\leq c_{\delta}\left(|\A_\theta(u,\eta)|+\|\nabla_J\A_\theta(u,\eta)\|_J+1\right).
\eea

\vspace{2mm}

There exists $\delta>0$ such that
\bean
\lambda(X_F(p))\geq\frac{1}{2}+\delta, \quad \forall\,p\in U_\delta
\eea
We compute
\bean
|\A_\theta(u,\eta)|
=&\left|\int_0^1u^*\lambda-\int_0^1H(t,u(t))dt-\eta\int_0^1F(t,u(t))dt+\Bc_\theta(u(t))\right| \\
=&\left|\int_0^1\widetilde u^*(\widetilde\lambda+\widetilde\tau^*\theta_t)-\int_0^1H(t,u(t))dt-\eta\int_0^1F(t,u(t))dt\right| \\
=&\bigg|\int_0^1(\widetilde\lambda+\widetilde\tau^*\theta_t)(\widetilde u(t))
[\frac{d}{dt}\widetilde u-\eta\widetilde X_F(t,\widetilde u)-\widetilde X_H^\sigma(t,\widetilde u)]dt \\
&+\eta\int_0^1\underbrace{\lambda(u(t))[X_F(t,u)]}_{\geq\frac{1}{2}+\delta}-\underbrace{F(t,u(t))}_{\leq\delta}dt \\
&+\int_0^1(\widetilde\lambda+\widetilde\tau^*\theta_t)(\widetilde u(t))[\widetilde X_H^\sigma(t,\widetilde u)]-H(t,u(t))dt\bigg| \\
\geq&\frac{1}{2}|\eta|-c_{\theta,\delta}' \|\frac{d}{dt}u-X_H^\sigma(t,u)-\eta X_F(t,u)\|_1-c(H,\theta) \\
\geq&\frac{1}{2}|\eta|-c_{\theta,\delta}' \|\frac{d}{dt}u-X_H^\sigma(t,u)-\eta X_F(t,u)\|_2-c(H,\theta) \\
\geq&\frac{1}{2}|\eta|-c_{\theta,\delta}' \|\nabla_J\A_\theta(u,\eta)\|_J-c(H,\theta) \\
\eea
Here $\widetilde{[-]}$ means the lifting of $[-]$ to the universal cover 
and $c_{\theta,\delta}':=\|(\widetilde\lambda+\widetilde\tau^*\theta)|_{\widetilde U_\delta}\|_{\infty}$.
Set $c_\delta=\max\{2c_{\theta,\delta}',2c(H,\theta),2\}$ then
this inequality proves Step 1.
Note that the finiteness of $c_{\theta,\delta}'$ is guaranteed by the simple estimate as follows
\bean
c_{\theta,\delta}'=&\|\widetilde\lambda+\tau^*\theta|_{\widetilde U_\delta}\|_{\infty} \\
\leq&\|\widetilde\lambda|_{\widetilde U_\delta}\|_\infty
+\|\widetilde\tau^*\theta|_{\widetilde U_\delta}\|_\infty \\
=&\|\lambda|_{U_\delta}\|_\infty+\|\theta|_{\widetilde\tau(\widetilde U_\delta)}\|_\infty \\
\leq&\|\lambda|_{U_\delta}\|_\infty+\|\theta\|_\infty \\
<&\infty.
\eea

\vspace{2mm}

{\bf Step 2} : {\em There exists $\epsilon=\epsilon(\delta)$ with the following property.
If there exists $t\in[0,\frac{1}{2}]$ with $F(t,u(t))\geq\delta$ then $\|\nabla_J\A_\theta(u,\eta)\|_J\geq\epsilon$.}

\vspace{2mm}

If in addition $F(t,u(t))\geq\frac{\delta}{2}$ holds for all $t\in[0,\frac{1}{2}]$ then
\bean
\|\nabla_J\A_{\theta(s)}(u,\eta)\|_J\geq\left|\int_0^1F(t,u(t))dt\right|\geq\frac{\delta}{2}\int_0^1\rho(t)dt=\frac{\delta}{2}.
\eea
Otherwise there exists $t'\in[0,\frac{1}{2}]$ with $F(t,u(t'))\leq\frac{\delta}{2}$. 
Thus we may assume without loss of generality that $0\leq a<b\leq\frac{1}{2}$ 
and $\frac{\delta}{2}\leq|F(t,u(t))|\leq\delta$ for all $t\in[a,b]$,
and $|F(b,u(b))-F(a,u(a))|=\frac{\delta}{2}$.
Then we estimate
\bean
\|\nabla_J\A_\theta(u,\eta)\|_J\geq&\|\frac{d}{dt}u-X_H^\sigma(t,u)-\eta X_F(t,u)\|_2 \\
\geq&\bigg(\int_a^b\|\frac{d}{dt}u-\underbrace{X_H^\sigma(t,u)}_{=0}-\eta X_F(t,u)\|^2dt\bigg)^{\frac{1}{2}} \\
\geq&\int_a^b\|\frac{d}{dt}u-\eta X_F(t,u)\|dt \\
\geq&\frac{1}{\|\nabla F\|_{\infty}}\int_a^b\|\nabla F(t,u(t))\|\cdot\|\frac{d}{dt}u-\eta X_F(t,u)\|dt \\
\geq&\frac{1}{\|\nabla F\|_{\infty}}\int_a^b\left|g_t\left(\nabla F(t,u(t)),\frac{d}{dt}u-\eta X_F(t,u)\right)\right|dt \\
=&\frac{1}{\|\nabla F\|_{\infty}}\int_a^b\left|g_t\left(\nabla F(t,u(t)),\frac{d}{dt}u\right)\right|dt \\
=&\frac{1}{\|\nabla F\|_{\infty}}\int_a^b\left|\frac{d}{dt}F(t,u(t))\right|dt \\
\geq&\frac{1}{\|\nabla F\|_{\infty}}\int_a^b\frac{d}{dt}F(t,u(t))dt \\
=&\frac{\delta}{2\|\nabla F\|_{\infty}}. \\
\eea
Since $\|\nabla F\|_{\infty}$ is bounded from above, 
we set $\epsilon(\delta):=\min\{\frac{\delta}{2},\frac{\delta}{2\|\nabla F\|_{\infty}}\}$. This proves Step 2.

\vspace{2mm}

{\bf Step 3} : {\em We prove the lemma.}

\vspace{2mm}

Choose $\delta$ as in Step 1, $\epsilon=\epsilon(\delta)$ as in Step 2 and 
\bean
\widetilde c=c_\delta(\epsilon+1).
\eea
Assume that $\|\nabla_J\A_\theta(u,\eta)\|_J\leq\epsilon$ then
\bean
|\eta|\leq c_{\delta}\left(|\A_\theta(u,\eta)|+\|\nabla_J\A_\theta(u,\eta)\|_J+1\right)
\leq\widetilde c(|\A_\theta(u,\eta)|+1).
\eea
This proves the lemma.
\end{proof}


\begin{Prop}\label{prop:etaesusu}
Let $w_\pm\in\Crit(\A_\theta)$ and $w=(u,\eta)$ be a gradient flow line of $\A_\theta$ with
\bean
\lim_{s\to\pm\infty}w(s)=w_\pm.
\eea
Then there exists a constant $\kappa=\kappa(w_-,w_+)$ satisfying
$\|\eta\|_\infty\leq\kappa.$
\end{Prop}

\begin{proof}
Let $\epsilon$ be as in Lemma \ref{lem:fund}. For $l\in\R$, let $\nu_w(l)\geq 0$ be defined by
\bean
\nu_w(l):=\inf\{\nu\geq0\,:\,\|\nabla_J\A_\theta[w(l+\nu)]\|_J<\epsilon\}.
\eea
Then $\nu_w(l)$ is uniformly bounded as follows,
\bean
\A_\theta(w_+)-\A_\theta(w_-)
&=\intinf\|\frac{d}{ds} w(s)\|^2ds \\
&=\intinf\|\nabla_J\A_\theta(w(s))\|_J^2 ds \\
&\geq\int_l^{l+\nu_w(l)}\underbrace{\|\nabla_J\A_\theta(w(s))\|_J^2}_{\geq\epsilon^2} ds \\
&\geq\epsilon^2\nu_w(l).
\eea

Now, we set
\bean
\|F\|_\infty=\max_{(t,x)\in \R/\Z\times T^*N}|F(t,x)|,\quad K=\max\{|\A_\theta(w_+)|,|\A_\theta(w_-)|\}.
\eea
Since $F\in\D(\Sigma)$, see Definition \ref{def:defham}, $\|F\|_\infty$ is finite.
By definition of $\nu_w(l)$, we get $\|\nabla_J\A_\theta[w(l+\nu_w(l))]\|_J=\epsilon$.
Then we can use Proposition \ref{lem:fund} to obtain the following estimate
\bean
|\eta(l+\nu_w(l))|&\leq\widetilde c(|\A_\theta[w(l+\nu_w(l))]|+1) \\
&\leq\widetilde c(K+1).
\eea
By using the above estimate we get
\bean
|\eta(l)|&\leq|\eta(l+\nu_w(l))|+\left|\int_l^{l+\nu_w(l)}\dot\eta(s)ds\right| \\
&\leq|\eta(l+\nu_w(l))|+\left|\int_l^{l+\nu_w(l)}\int_0^1F(t,u(t))dt\ ds\right| \\
&\leq\widetilde c(K+1)+\|F\|_\infty\nu_w(l) \\
&\leq\widetilde c(K+1)+\frac{\|F\|_\infty(\A_\theta(w_+)-\A_\theta(w_-))}{\epsilon^2}.
\eea
The right hand side is independent of the gradient flow line $w$ and $l\in\R$.
Let 
\bean
\kappa=\widetilde c(K+1)+\frac{\|F\|_\infty(\A_\theta(w_+)-\A_\theta(w_-))}{\epsilon^2},
\eea
then this proves the proposition.
\end{proof}

\begin{Prop}\label{prop:asphe}
If $\sigma_0\in\Om^2(N)$ is $\widetilde d$-bounded then $(\om_{\rm std}+\tau^*\sigma_0)|_{\pi_2(T^*N)}=0$.
\end{Prop}

\begin{proof}
First choose a map $f:S^2\to T^*N$, then it suffices to show that $\int_{f(S^2)}\om_{\rm std}+\tau^*\sigma_0=0$.
\bean
\int_{f(S^2)}\om_{\rm std}+\tau^*\sigma_0
&=\int_{f(S^2)}\om_{\rm std}+\int_{\tau\circ f(S^2)}\sigma_0
=\int_{f(S^2)}d\lambda+\int_{\widetilde{\tau\circ f}(S^2)}\widetilde\sigma_0 \\ 
&=\int_{\p f(S^2)}\lambda+\int_{\widetilde{\tau\circ f}(S^2)}d\theta_0 
=\int_{\p\widetilde{\tau\circ f}(S^2)}\theta_0 \\
&=0.
\eea
Here $\widetilde{\tau\circ f}:S^2\to\widetilde N$ 
is a lifting of $\tau\circ f:S^2\to N$.
\end{proof}

\begin{Thm}\label{thm:compofmod}
Let $w^\nu=(u^\nu,\eta^\nu)$ be a sequence of gradient flow lines for which there exists $a<b$ such that
\[
a\leq\A_\theta\big(w^\nu(s)\big)\leq b,\quad\forall s\in\R.
\]
Then the sequence $\{w^\nu\}$ has a subsequence 
that converges in $C^\infty_{\rm loc}(\R,\L\times\R)$.
\end{Thm}

\begin{proof}
As we mentioned in Remark \ref{rmk:compact},
we establish compactness by proving the following uniform bounds:
\begin{enumerate}
\item a uniform $L^\infty$ bound on $u^\nu$;
\item a uniform $L^\infty$ bound on $\eta^\nu$;
\item a uniform $L^\infty$ bound on the derivatives of $u^\nu$.
\end{enumerate} 
Property (1) follows from the convexity at infinity of $(T^*N,\om_\sigma)$.
Proposition \ref{prop:etaesusu} implies that $\eta^\nu$ is uniformly bounded
which guarantees property (2).
If the derivatives would explode we then obtain a non-constant holomorphic sphere as limit
which is impossible by Proposition \ref{prop:asphe}.
Hence we conclude property (3).
With (1)-(3), the bootstrapping argument gives $C^\infty_{\rm loc}$-convergence,
see \cite[Appendix B.4]{MS}.
\end{proof}

\begin{Thm}\label{thm:welldefAm}
Let $\Sigma\subset T^*N$ be a closed hypersurface with a defining Hamiltonian $F$.
For a generic choice of $(H,\theta)$, the functional $\A_\theta$ is Morse 
and $\FH(\A_\theta)$ is well-defined.
\end{Thm}

\begin{proof}
Now we are ready to define the Floer homology of $\A_\theta$.
By choosing a generic pair $(H,\theta)\in\Hc\times\mathcal P$,
we may assume that $\A_\theta$ is Morse.
Take a critical point $w=(u,\eta)$ of $\A_\theta$.
Since $\A_\theta$ is Morse, 
$u:\R/\Z\to T^*N$ is a non-degenerate orbit.
Thus we can associate to $u$ a well-defined integer 
the {\em Conley-Zehnder index} $\CZ(u)$. 
See \cite{SZ} and \cite{AS06} for the definition and
details of the Conley-Zehnder index.

Let us define $\mu(w):=\CZ(u)$ and 
denote by 
\bean
\Crit^{(a,b)}(\A_\theta)&:=\{(u,\eta)\in\Crit(\A_\theta)\,:\,\A_\theta(u,\eta)\in(a,b)\};\\
\Crit^{(a,b)}_k(\A_\theta)&:=\{w\in\Crit^{(a,b)}(\A_\theta)\,:\, \mu(w)=k\}.
\eea
Now we define
\bean
\FC_k^{(a,b)}(\A_\theta):=\Crit_k^{(a,b)}(\A_\theta)\otimes\Z_2.
\eea
For a generic almost complex structure $J(t)\in\J_\sigma$ 
and given $w_\pm\in\Crit^{(a,b)}(\A_\theta)$, we denote by
\bean
\widehat\M(w_-,w_+)&:=\{w(s)\,:\,w\text{ satisfies }(\ref{eqn:gradeqnm}),\ \lim_{s\to\pm\infty}w(s)=w_\pm\};\\
\M(w_-,w_+)&:=\widehat\M(w_-,w_+)/\R.
\eea
The above $\R$-action is given by translating the $s$-coordinate.
Suppose further that the almost complex structure $J(t)$ is generic,
so that $\M(w_-,w_+)$ is a smooth manifold of dimension
\bean
\dim\M(w_-,w_+)=\mu(w_-)-\mu(w_+)-1.
\eea

For the compactification of the moduli space $\widehat\M(w_-,w_+)$,
we remind the {\em Floer-Gromov convergence} as follows.
A sequence $\{(u^\nu,\eta^\nu)\}_{\nu\in\N}$ in $\widehat\M(w_-,w_+)$ is said to  
Floer-Gromov converges to a broken flow line $\{(u_j,\eta_j)\}_{j=1}^m$ with
\bean
(u_j,\eta_j)\in\widehat\M(w_{j-1},w_j),\quad j\in\{1,\dots,m\},
\eea
where $w_i\in\Crit(\A_\theta)$ and $w_0=w_-$, $w_m=w_+$,
if there exist $s_j^\nu\in\R$ such that reparametrized 
sequence $(u^\nu,\eta^\nu)(s+s_j^\nu)$ converges to $(u_j,\eta_j)(s)$
for all $j=1,\dots,m$ in the $C^\infty_{\rm loc}$-topology.
By Theorem \ref{thm:compofmod}, we are now able to consider 
a compactification of $\widehat\M(w_-,w_+)$ with respect to 
the Floer-Gromov convergence.
We mention two following statements, due to Floer \cite{Fl2,Fl1}, which is crucial in constructing
various type of Floer homologies, also containing $\A_\theta$-version. 
\begin{enumerate}
\item If $\mu(w_-)-\mu(w_+)=1$, then
the moduli space $\widehat\M(w_-,w_+)$ is a 1-dimensional compact
smooth manifold with respect to the Floer-Gromov convergence and
hence $\M(w_-,w_+)$ is finite.
\item Let $\overline\M(w_-,w_+)$ be the compactification of 
$\M(w_-,w_+)$ with respect to the Floer-Gromov convergence. 
If $\mu(w_-)-\mu(w_+)=2$, then $\overline\M(w_-,w_+)$ is a compact 
1-dimensional manifold whose boundary is given by
\[
\p\overline\M(w_-,w_+)=\bigcup_{w}\M(w_-,w)\times\M(w,w_+),
\]
where $w\in\Crit(\A_\theta)$ runs over $\mu(w_-)-\mu(w)=1$.
\end{enumerate}

By virtue of (1), the boundary operator $\p_k:\FC_k^{(a,b)}(\A_\theta)\to\FC^{(a,b)}_{k-1}(\A_\theta)$
is defined by
\bean
\p_k w_-:=\sum_{\mu(w_+)=k-1}\#_2\M(w_-,w_+)\cdot w_+,
\eea
where $\#_2$ means $\Z_2$-counting.
This boundary operator satisfies $\p_{k-1}\circ\p_k=0$ in $\Z_2$-coefficient
which is guaranteed by (2).
Then the resulting filtered Floer homology group is defined by
\bean
\FH^{(a,b)}_*(\A_\theta):=\H_*(\FC^{(a,b)}_\bullet(\A_\theta),\p_*).
\eea
By taking direct and inverse limit, we obtain
\bean
\FH_*(\A_\theta):=\lim_{\stackrel\longrightarrow a} \lim_{\stackrel \longleftarrow
b}\FH_*^{(-a,b)}(\A_\theta), \qquad a,b\to\infty.
\eea
\end{proof}

\subsection{Continuation map between $\FH(\A^F_H)$ and $\FH(\A_\theta)$.}
In this section, we construct a continuation homomorphism
\bean
\widetilde\Psi^\sigma:\FH(\A^F_H)\to\FH(\A_\theta)
\eea
by counting gradient flow lines of the $s$-dependent perturbed Rabinowitz action functional.
Here 
\bea\label{eqn:m(s)def}
\sigma(s):=\gamma(s)\sigma,\qquad
\theta(s):=\gamma(s)\theta
\eea
are $s$-dependent magnetic perturbation
where $\gamma:\R\to[0,1]$ is a cut-off function 
\bea\label{eqn:gamma(s)}
\gamma(s)=
\left\{
\begin{array}{cc}
0\quad \text{for} \quad s\leq 0 \\
1\quad \text{for} \quad s\geq 1
\end{array}
\right.
\eea
and $0\leq\dot\gamma(s)\leq2$ for all $s\in\R$.
Then the corresponding action functional is
\bean
\A_{\theta(s)}(u,\eta)=\A^F_H(u,\eta)+\gamma(s)\Bc_\theta(u)
\eea
where $\A^F_H(u,\eta)=\int_0^1u^*\lambda-\eta\int_0^1F(t,u(t))dt-\int_0^1H(t,u(t))dt$.

Now we consider the $(s,t)$-dependent almost complex structure $J(s,t)$ on $T^*N$ such that
$J(s,t)\in\J_{\sigma(s)}$ for all $s\in[0,1]$ and $J(s,t)$ is independent of $s$ for $s\leq-1$
and $s\geq1$. This almost complex structure induces the $(s,t)$-dependent inner product on $T^*N$
\bean
g_{s,t}(\,\cdot\,,\,\cdot\,):=\om_{\sigma(s)}(\,\cdot\,,J(s,t)\,\cdot\,),
\eea
and the following $s$-dependent inner product on $\L\times\R$
\bea\label{eqn:m(s)metric}
\mathfrak g_s\big((\hat u_1,\hat\eta_1),(\hat u_2,\hat\eta_2)\big)
:=\int_0^1 g_{s,t}(\hat u_1,\hat u_2)dt+\hat\eta_1\hat\eta_2,
\eea
where $(\hat u_i,\hat\eta_i)\in T_{(u,\eta)}(\L\times\R)$ for $i=1,2$.
With the above metric, we obtain
\bean
\nabla_s\A_{\theta(s)}(u,\eta)=
\left(
\begin{array}{cc}
-J(s,t,u)\big(\p_t u-\eta X_F(t,u)-X_H^{\sigma(s)}(t,u)\big) \\
-\int_0^1F(t,u)dt.
\end{array}
\right)
\eea
Then the gradient flow line $w=(u,\eta)\in C^\infty(\R\times (\R/\Z),T^*N)\times C^\infty(\R,\R)$ satisfies
\bea\label{eqn:grad_m(s)}
\left.
\begin{array}{cc}
\p_su+J(s,t,u)\big(\p_t u-\eta X_F(t,u)-X_H^{\sigma(s)}(t,u)\big)=0 \\
\frac{d}{ds}\eta+\int_0^1F(t,u)dt=0,
\end{array}
\right\}
\eea
with energy
\bean
E(w):=\intinf\|\frac{d}{ds}w(s)\|_s^2ds,
\eea
where $\|\cdot\|_s:=\sqrt{\mathfrak g_s(\cdot,\cdot)}.$

In order to construct the continuation homomorphism $\widetilde\Psi^\sigma$,
it suffices to show that the Lagrange multiplier $\eta$
and the energy of the time-dependent gradient flow line are uniformly bounded.
For this purpose, we need the following preliminary statements.

\begin{Lemma}\label{lem:fund2}
There exists $\epsilon>0$ and $\widetilde c>0$ such that 
if $(u,\eta)\in C^\infty(\R/\Z,T^*N)\times\R$ satisfies
$\|\nabla_s\A_{\theta(s)}(u,\eta)\|_s<\epsilon$ then
\beq\label{eqn:fund2}
|\eta|\leq\widetilde c(|\A_{\theta(s)}(u,\eta)|+1).
\eeq
Here $\|\cdot\|_s:=\sqrt{\mathfrak g_s(\cdot,\cdot)}.$
\end{Lemma}

\begin{proof}
The proof is basically the same as in Lemma \ref{lem:fund} 
by considering $\sigma(s)$ instead of $\sigma$.
Here we omit the proof.
\end{proof}

\begin{Def}
For a magnetic perturbation $\sigma\in\Mf$ 
and its primitive $\theta\in\mathcal P$, see Definition \ref{def:betadef}, 
the {\em isoperimetric constant} $C=C(\theta)$ is defined by
\bean
C:=\|\theta\|_\infty=\max_{t\in \R/\Z} \|\theta_t\|_\infty.
\eea
\end{Def}

In order to state the next proposition, we introduce the notations as follows:
\bean
d_\sigma=&d_{H,\sigma}:=\sup_{s\in\R}\|X_H^{\sigma(s)}\|_{\infty};\\
d_{F}:=&\|X_F\|_{\infty}.
\eea

\begin{Prop}\label{prop:conti}
Let $\epsilon$, $\widetilde c$ be the same as in Lemma \ref{lem:fund2}.
Let $w_-\in\Crit(\A^F_H)$, $w_+\in\Crit(\A_\theta)$ and $w=(u,\eta)$ be a gradient flow line of $\A_{\theta(s)}$ with
$\lim_{s\to\pm\infty}w=w_{\pm}$. 
If the isoperimetric constant $C=C(\theta)$ satisfies
\bea\label{eqn:Ccond}
C&\leq\frac{1}{4}; \\
\left(8\widetilde cd_{F}C+2\widetilde cd_{F}+4\frac{d_{F}}{\epsilon^2}\|F\|_{\infty}\right)C&\leq\frac{1}{2},
\eea
then there exists a constant $\kappa=\kappa(w_-,w_+)$ such that
\bean
\|\eta\|_{\infty}\leq\kappa.
\eea
\end{Prop}

\begin{proof} 
We prove the proposition in 3 steps.

\vspace{2mm}

{\bf Step 1} : {\em Let us first bound the energy of $w$ in terms of $\|\eta\|_{\infty}$.}

\vspace{2mm}

\bea\label{eqn:EesdotA}
E(w)=&\intinf\|\frac{d}{ds}w(s)\|_s^2ds \\
=&\intinf\langle\frac{d}{ds}w(s),\nabla_s\A_{\theta(s)}(w(s))\rangle_sds \\
=&\intinf\frac{d}{ds}\A_{\theta(s)}(w(s))ds-\intinf\dot\A_{\theta(s)}(w(s))ds \\
=&\A_{\theta(1)}(w_+)-\A_{\theta(0)}(w_-)-\intinf\dot\A_{\theta(s)}(w(s))ds.
\eea
We estimate the last term in (\ref{eqn:EesdotA}) by using the isoperimetric constant $C$
\bea\label{eqn:dotAesC}
\left|\intinf\dot\A_{\theta(s)}(w(s))ds\right|\leq&\intinf\left|\dot\A_{\theta(s)}(w(s))\right|ds \\
=&\intinf\dot\gamma(s)\left|\int_{\R/\Z}\widetilde u^*\theta_t dt\right|ds \\
\leq&\intinf\dot\gamma(s)C\int_{\R/\Z}|\p_t u|_{s,t}dt\ ds.
\eea
Here $|\cdot|_{s,t}:=\sqrt{g_{s,t}(\cdot,\cdot)}$.
From the gradient flow equation (\ref{eqn:grad_m(s)}), we get
\bea\label{eqn:vargradeqn}
\p_t u=J(s,t,u)\p_su+\eta X_F(t,u)+X_H^{\sigma(s)}(t,u).
\eea
By inserting (\ref{eqn:vargradeqn}) into the last term in (\ref{eqn:dotAesC}), we then obtain
\bea\label{eqn:dotA}
\intinf\left|\dot\A_{\theta(s)}(w(s))\right|ds\leq&\intinf\dot\gamma(s)C\int_{\R/\Z}|\p_t u|_{s,t}dt\ ds \\
=&\intinf\underbrace{\dot\gamma(s)}_{\leq2}C\int_{\R/\Z}|J(s,t)\p_su+\eta X_F(t,u)+X_H^{\sigma(s)}(t,u)|_{s,t}dt\ ds \\
\leq&2C\int_0^1\int_{\R/\Z}\left(|\p_su|_{s,t}+|\eta| |X_F(t,u)|_{s,t}+|X_H^{\sigma(s)}(t,u)|_{s,t}\right)dt\ ds \\
\leq&2C\int_0^1\int_{\R/\Z}\left(|\p_su|^2_{s,t}+1+|\eta| \|X_F\|_{\infty}+\|X_H^{\sigma(s)}\|_{\infty}\right)dt\ ds \\
=&2CE(u)+2C+2d_\sigma C+2\|\eta\|_{\infty} d_{F}C \\
\leq&2CE(w)+2C+2d_\sigma C+2\|\eta\|_{\infty} d_{F}C.
\eea

Now by combining the above estimates (\ref{eqn:EesdotA}) and (\ref{eqn:dotA}), 
we deduce
\bean
E(w)=&\A_{\theta(1)}(w_+)-\A_{\theta(0)}(w_-)-\intinf\dot\A_{\theta(s)}(w(s))ds \\
\leq&\A_{\theta(1)}(w_+)-\A_{\theta(0)}(w_-)+2CE(w)+2C+2d_\sigma C+2\|\eta\|_{\infty} d_{F}C.
\eea
By the assumption on the isoperimetric constant, $C\leq\frac{1}{4}$ and we have
\bea\label{eqn:E<eta}
E(w)\leq& 2\A_{\theta(1)}(w_+)-2\A_{\theta(0)}(w_-)+4C+4d_\sigma C+4\|\eta\|_{\infty} d_{F}C \\
=&2\Delta+4C+4d_\sigma C+4\|\eta\|_{\infty} d_{F}C,
\eea
where $\Delta:=\A_{\theta(1)}(w_+)-\A_{\theta(0)}(w_-)$.
This finishes Step 1.

\vspace{2mm}

{\bf Step 2} : Let $\epsilon$ be as in Lemma \ref{lem:fund}. For $l\in\R$ let $\nu_w(l)\geq0$ be defined by
\bean
\nu_w(l):=\inf\left\{\nu\geq0:\|\nabla_s\A_{\theta(l+\nu)}\big(w(l+\nu)\big)\|_s<\epsilon\right\}.
\eea
In this step we bound $\nu_w(l)$ in terms of $\|\eta\|_{\infty}$ for all $l\in R$ as follows
\bea\label{eqn:nuesenergy}
E(w)=&\intinf\|\frac{d}{ds}w(s)\|^2_sds \\
=&\intinf\|\nabla_s\A_{\theta(s)}\|^2_sds \\
\geq&\int_l^{l+\nu_w(l)}\underbrace{\|\nabla_s\A_{\theta(s)}\|^2_s}_{\geq\epsilon^2}ds \\
\geq&\epsilon^2\nu_w(l).
\eea
Step 1 and the above estimate finish Step 2.

\vspace{2mm}

{\bf Step 3} : {\em We prove the proposition.} \\

\vspace{2mm}

First set
\[
K=\max\{-\A_{\theta(0)}(w_-),\A_{\theta(1)}(w_+)\}
\]
By the definition of $\nu_w(l)$, we get $\|\nabla_s\A_{\theta[l+\nu_w(l)]}[w(l+\nu_w(l))]\|_s<\epsilon$,
which enables us to use Lemma \ref{lem:fund2}.
We obtain the following estimate by using (\ref{eqn:fund2}), (\ref{eqn:dotA}) and (\ref{eqn:E<eta})
\bea\label{eqn:eta1}
|\eta(l+\nu_w(l))|\leq&\widetilde c\left(\left|\A_{\theta[l+\nu_w(l)]}[w(l+\nu_w(l))]\right|+1\right) \\
\leq&\widetilde c\left(K+\intinf\left|\dot\A_{\theta(s)}\right|ds+1\right) \\
\leq&\widetilde c\left(K+2CE(w)+2C+2d_\sigma C+2\|\eta\|_{\infty} d_{F}C+1\right) \\
\leq&\widetilde c\big[K+2C(2\Delta+4C+4d_\sigma C+4\|\eta\|_{\infty} d_{F}C) \\
&+2C+2d_\sigma C+2\|\eta\|_{\infty} d_{F}C+1\big]. \\
\eea
By Step 2 and (\ref{eqn:E<eta}), we get the following inequalities
\bea\label{eqn:eta2}
\left|\int_l^{l+\nu_w(l)}\dot\eta(s)ds\right|\leq&\left|\int_l^{l+\nu_w(l)}\int_0^1F(t,u(t))dt\ ds\right| \\
\leq&\|F\|_{\infty}\nu_w(l) \\
\leq&\|F\|_{\infty}\frac{E(w)}{\epsilon^2} \\
\leq&\frac{\|F\|_{\infty}}{\epsilon^2}(2\Delta+4C+4d_\sigma C+4\|\eta\|_{\infty} d_{F}C). \\
\eea
Combining the above two estimates (\ref{eqn:eta1}) and (\ref{eqn:eta2}), 
we conclude the following
\bean
|\eta(l)|\leq&|\eta(l+\nu_w(l))|+\left|\int_l^{l+\nu_w(l)}\dot\eta(s)ds\right| \\
\leq&\widetilde c\big[K+2C(2\Delta+4C+4d_\sigma C+4\|\eta\|_{\infty} d_{F}C)+2C+2d_\sigma C+2\|\eta\|_{\infty} d_{F}C+1\big] \\
&+\frac{\|F\|_{\infty}}{\epsilon^2}(2\Delta+4C+4d_\sigma C+4\|\eta\|_{\infty} d_{F}C) \\
=&\underbrace{\left(8\widetilde cd_{F}C+2\widetilde cd_{F}+4\frac{d_{F}}{\epsilon^2}\|F\|_{\infty}\right)C}_{=:\kappa_1}\|\eta\|_{\infty} \\
&+\underbrace{\widetilde c\big[K+2C(2\Delta+4C+4d_\sigma C)+2C+2d_\sigma C+1\big]+\frac{\|F\|_{\infty}}{\epsilon^2}(2\Delta+4C+4d_\sigma C)}_{=:\frac{1}{2}\kappa}. \\
\eea
Note that the above estimate is valid for all $l\in\R$ and $\kappa_1\leq\frac{1}{2}$ by the condition (\ref{eqn:Ccond}).
Thus we have 
\bean
\|\eta\|_{\infty}\leq\frac{1}{2}\|\eta\|_{\infty}+\frac{1}{2}\kappa,
\eea
and hance we conclude
\beq\label{eqn:etabdfin}
\|\eta\|_{\infty}\leq\kappa.
\eeq
\end{proof}

\begin{Lemma}\label{lem:actiones}
Let $\epsilon$, $\widetilde c$ in (\ref{eqn:Ccondi2}) be the same as in Lemma \ref{lem:fund2}.
Let $w_-\in\Crit(\A^F_H)$, $w_+\in\Crit(\A_\theta)$ and $w=(u,\eta)$ 
be a gradient flow line of $\A_{\theta(s)}$ with
$\lim_{s\to\pm\infty}w=w_{\pm}$ and its action value
$a=\A_{\theta(0)}(w_-)$, $b=\A_{\theta(1)}(w_+)$.
If the isoperimetric constant $C=C(\theta)$ satisfies the following conditions:
\bea\label{eqn:Ccondi2}
\widetilde c\,d_F\,C&\leq\ \frac{1}{32}; \\
\left(2\widetilde c\,d_F\,C+\frac{d_F\|F\|_\infty}{\epsilon}\right)C&\leq\frac{1}{128}; \\
\bigg(1+d_\sigma +2\,\widetilde c\,d_F+4\,\widetilde c\,d_F\,C(1+d_\sigma +4\,C+4\,d_\sigma C)& \\
+\frac{8\,d_F\|F\|_\infty C}{\epsilon^2}(1+d_\sigma )\bigg)C&\leq\ \,\frac{1}{72},
\eea
then the following assertions hold:
\begin{enumerate}
\item If $a\geq\frac{1}{9}$, then $b\geq\frac{a}{2}$;
\item If $b\leq-\frac{1}{9}$, then $a\leq\frac{b}{2}$.
\end{enumerate}
\end{Lemma}

\begin{proof}
By Proposition \ref{prop:conti}, the Lagrange multiplier $\eta$ 
of the gradient flow line is uniformly bounded as follows
\bea\label{eqn:etainftyesti}
\|\eta\|_\infty\leq2\widetilde c\big(K+2C(2\Delta+4C+4d_\sigma C)+2C+2d_\sigma C+1\big)+\frac{2\|F\|_{\infty}}{\epsilon^2}(2\Delta+4C+4d_\sigma C).
\eea
Recall that $K=\max\{-a,b\}$ and $\Delta=b-a$.
From the fact that $E(w)\geq0$ and (\ref{eqn:E<eta}), we obtain the following inequality
\bea\label{eqn:actionestimate}
b\geq a-2C-2d_\sigma C-2\|\eta\|_\infty d_F C.
\eea 
By combining (\ref{eqn:etainftyesti}) with (\ref{eqn:actionestimate}), we obtain
\bea\label{eqn:aces_ab}
b
\geq& a-2C-2d_\sigma C-2\|\eta\|_\infty d_F C \\
\geq& a-2C-2d_\sigma C-4 \widetilde c\big(K+2C(2\Delta+4C+4d_\sigma C)+2C+2d_\sigma C+1\big)d_F C \\
&+\frac{4\|F\|_{\infty}}{\epsilon^2}(2\Delta+4C+4d_\sigma C) d_F C \\
=& a-4\,\widetilde c\,d_F\,C\,K-8\left(2\widetilde c\,d_F\,C+\frac{d_F\|F\|_\infty}{\epsilon}\right)C\,\Delta \\
&-2\bigg(\frac{8\,d_F\|F\|_\infty C}{\epsilon^2}(1+d_\sigma )
+1+d_\sigma +2\,\widetilde c\,d_F+4\,\widetilde c\,d_F\,C(1+d_\sigma +4\,C+4\,d_\sigma C) \bigg)C \\
\geq&a-\frac{1}{8}K-\frac{1}{16}(b-a)-\frac{1}{36}.
\eea
Here the last inequality we use the condition (\ref{eqn:Ccondi2}).
To prove the assertion (1), we first consider the case
\bean
|b|\leq a, \qquad a\geq\frac{1}{9}.
\eea
In this assumption, we induce the following estimate from (\ref{eqn:aces_ab})
\bean
b\geq a-\frac{1}{8}a-\frac{1}{8}a-\frac{1}{36}=\frac{3}{4}a-\frac{1}{36}\geq\frac{a}{2}.
\eea
Now we want to exclude the case
\bean
-b\geq a \geq\frac{1}{9}.
\eea
But in this case (\ref{eqn:aces_ab}) implies the following contradiction:
\bean
b\geq\frac{1}{9}+\frac{1}{72}-\frac{1}{16}(b-a)-\frac{1}{36}\geq-\frac{1}{16}(b-a)>0.
\eea
This proves the first assumption. To prove the assertion (2), we set
\bean
b'=-a,\qquad a'=-b.
\eea
Then (\ref{eqn:aces_ab}) also holds for $b'$ and $a'$. Thus we get the following assertion from (1)
\bean
-b\geq\frac{1}{9}\Longrightarrow -a\geq-\frac{b}{2}
\eea
which is equivalent to the assertion (2). This proves the lemma.
\end{proof}

\begin{Thm}\label{thm:continuation}
Let $\Sigma\subset T^*N$ be a closed hypersurface with a defining Hamiltonian $F$.
For a generic choice of $(H,\theta)$, the functionals $\A^F_H$, $\A_\theta$ are Morse and
\bean
\FH(\A^F_H)\cong\FH(\A_\theta).
\eea
\end{Thm}

\begin{proof}
For a given magnetic perturbation $\sigma\in\Mf$
and its primitive $\theta\in\mathcal P$,
we first consider the sequence $\{\theta^i\}_{i=1}^N\subset\mathcal P$
which satisfies the following properties:
\begin{enumerate}
\item $\theta^i:=d^i\theta$, where $0=d^0<d^1<\cdots<d^N=1$;
\item $\A_{\theta^i}:\L\times\R\to\R$ is Morse for all $i=0,1,\dots,N$;
\item $C^i:=(d^{i+1}-d^i)\|\theta\|_\infty$ satisfies 
the assumption of Proposition \ref{prop:conti} and Lemma \ref{lem:actiones} for all $i=0,1,\dots,N-1$.\footnote
{Since the Morse property of $\A_{\theta^i}$ is generic, we can guarantee property (2).
Note that we make $C^i$ arbitrarily small by choosing small $d^{i+1}-d^i$.}
\end{enumerate}

Now we are ready to use Proposition \ref{prop:conti} and Lemma \ref{lem:actiones} inductively.
Let us consider following homotopies 
\bean
\theta^i(s)&:=\theta^i+\gamma(s)(\theta^{i+1}-\theta^i);\\
\sigma^i(s)&:=\sigma^i+\gamma(s)(\sigma^{i+1}-\sigma^i).
\eea
Here $\widetilde\sigma^i=d\theta^i$ and see (\ref{eqn:gamma(s)}) for $\gamma(s)$,
and a gradient flow line $w=(u,\eta)\in C^\infty(\R\times (\R/\Z),T^*N)\times C^\infty(\R,\R)$ 
of the time dependent action functional $\A_{\theta^i(s)}$ which satisfies
\bea\label{eqn:gradeqnmi}
\left.
\begin{array}{cc}
\p_su+J(s,t,u)\big(\p_t u-X_H^{\sigma^i(s)}(t,u)-\eta X_F(t,u)\big)=0 \\
\frac{d}{ds}\eta+\int_0^1F(t,u)dt=0,
\end{array}
\right\}
\eea
and the limit condition
\bea\label{eqn:limconmi}
\lim_{s\to-\infty}w(s)=w_-\in\Crit(\A_{\theta^i}),\qquad
\lim_{s\to\infty}w(s)=w_+\in\Crit(\A_{\theta^{i+1}}).
\eea
The solutions of (\ref{eqn:gradeqnmi}), (\ref{eqn:limconmi}) form a moduli space
\bean
\Nc_i^{i+1}(w_-,w_+)=\{w=(u,\eta)\,:
\,(u,\eta)\text{ satisfies }(\ref{eqn:gradeqnmi}),\ (\ref{eqn:limconmi})\}
\eea
which, for a generic homotopy $\sigma^i(s)$, is a smooth manifold of dimension $\mu(w_-)-\mu(w_+)$.

For the compactness of $\Nc_i^{i+1}(w_-,w_+)$ we need to consider the analytic issues
about uniform bounds on $u$, $\eta$, the derivatives of $u$ and the energy of $w$.
The arguments in Theorem \ref{thm:compofmod} are also valid for 
the case of the uniform bounds on $u$ and its derivatives.
Proposition \ref{prop:conti} guarantees the uniform bound on $\eta$. 
Finally (\ref{eqn:E<eta}) with Proposition \ref{prop:conti}
implies the uniform energy bound of time-dependent gradient flow lines as follows,
\bean
E(w)
\leq& 2\A_{\theta^{i+1}(1)}(w_+)-2\A_{\theta^i(0)}(w_-)+4\,C+4d_{\sigma^{i+1}} C+4\|\eta\|_{\infty} d_{F}\,C \\
\leq& 2\A_{\theta^{i+1}(1)}(w_+)-2\A_{\theta^i(0)}(w_-)+4\,C+4d_{\sigma^{i+1}} C+4\,\kappa(w_-,w_+)\,d_{F}\,C.
\eea
We then consider the compactified moduli space $\overline\Nc_i^{i+1}(w_-,w_+)$ 
of $\Nc_i^{i+1}(w_-,w_+)$ with respect to 
Floer-Gromov convergence as in Theorem \ref{thm:welldefAm}.
When $\mu(w_-)-\mu(w_+)=1$ we have
\bea\label{eqn:chaincomm}
\p\overline\Nc_i^{i+1}(w_-,w_+) 
=\bigcup_x\M_{i}(w_-,x)\times\Nc_i^{i+1}(x,w_+)\cup
\bigcup_y\Nc_i^{i+1}(w_-,y)\times\M_{i+1}(y,w_+),
\eea
where $x,y$ run over $\Crit_{\mu(w_+)}(\A_{\theta^i})$, $\Crit_{\mu(w_-)}(\A_{\theta^{i+1}})$ respectively.

Now, by virtue of Lemma \ref{lem:actiones}, we define a map for $a\leq-\frac{1}{9}$ and $b\geq\frac{1}{9}$
\bean
\Psi^{(a,b)}_{i,k}:\FC_k^{(\frac{a}{2},b)}(\A_{\theta^i})\to \FC_k^{(a,\frac{b}{2})}(\A_{\theta^{i+1}})
\eea
given by
\bean
\Psi^{(a,b)}_{i,k}(w_-)=\sum_{\mu(w_+)=\mu(w_-)}\#_2\Nc_i^{i+1}(w_-,w_+)\ w_+.
\eea
Here $\#_2$ means the $\Z_2$-counting.

Now (\ref{eqn:chaincomm}) gives us $\Psi_{i,k-1}^{(a,b)}\circ\p_k=\p_k\circ\Psi_{i,k}^{(a,b)}$
in $\Z_2$ coefficient which implies that $\Psi_{i,k}^{(a,b)}$ is a chain map. 
Hence we have the following homomorphisms on homologies as follows
\bean
\widetilde\Psi^{(a,b)}_{i}:\FH^{(\frac{a}{2},b)}(\A_{\theta^i})\to \FH^{(a,\frac{b}{2})}(\A_{\theta^{i+1}}).
\eea
By taking the inverse and direct limit
\bean
\FH_*(\A_{\theta^i})=\lim_{b\to\infty}\lim_{a\to-\infty}\FH_*^{(a,b)}(\A_{\theta^i}),
\eea
we deduce
\bean
\widetilde\Psi_i:\FH(\A_{\theta^i})\to \FH(\A_{\theta^{i+1}}).
\eea
We define
\bean
\widetilde\Psi^\sigma:\FH(\A^F_H)\to \FH(\A_\theta).
\eea
by
$\widetilde\Psi^\sigma=\widetilde\Psi_{N-1}\circ\cdots\circ\widetilde\Psi_1\circ\widetilde\Psi_0$.
In a similar way, we construct
\bean
\widetilde\Psi_\sigma:\FH(\A_\theta)\to \FH(\A^F_H),
\eea
by following the homotopies in opposite direction.
By a homotopy-of-homotopies argument, we conclude 
$\widetilde\Psi_\sigma\circ\widetilde\Psi^\sigma=\id_{\FH(\A^F_H)}$
and $\widetilde\Psi^\sigma\circ\widetilde\Psi_\sigma=\id_{\FH(\A_\theta)}$. 
Therefore $\widetilde\Psi^\sigma$ is an isomorphism with inverse $\widetilde\Psi_\sigma$.
\end{proof}

By these results, if $\dim\H_*(\L_N)=\infty$ then we have infinitely many critical points of $\A_\theta$.
This implies that there are infinitely many magnetic leaf-wise intersections or
a {\em periodic} one which means that the leaf on which it lies forms a closed Reeb orbit.
We exclude the latter case generically, as follows. 

We recall a hypersurface $\Sigma\subset T^*N$ {\em non-degenerate} 
if closed Reeb orbits on $\Sigma$ form a discrete set.
A generic $\Sigma$ is non-degenerate, see \cite[Theorem B.1]{CF09}.
If $\Sigma$ is non-degenerate, 
then periodic leaf-wise intersection points
can be excluded by choosing a generic Hamiltonian function, 
see \cite[Theorem 3.3]{AF09b}.
With the above generic Hamiltonian,
Albers-Frauenfelder conclude that there are 
infinitely many leaf-wise intersection points on $\Sigma$,
under the topological assumption $\dim\H_*(\L_N)=\infty$.
By these reason, we only consider {\em non}-periodic 
(magnetic) leaf-wise intersection points. 
Thus we conclude the following existence result for
magnetic leaf-wise intersections.

\begin{Cor}\label{cor:infmlwip}
Let $N$ be a closed connected orientable manifold of dimension $n\geq2$.
Let $\Sigma$ be a non-degenerate hypersurface in $T^*N$.
Suppose that $\dim\H_*(\L_N)=\infty$. 
If $\varphi$ 
is generic then there exist infinitely many magnetic leaf-wise intersection points.
\end{Cor}

\begin{proof}[Proof of Corollary \ref{cor:infmlwip}]
In Theorem \ref{thm:continuation}, we have the continuation isomorphism as follows
\bea
\widetilde\Psi^\sigma:\FH_*(\A^F_H)\to\FH_*(\A_\theta).
\eea
Since we assume that $\dim\H_*(\L_N)=\infty$,
(\ref{eqn:RFHloop}), (\ref{eqn:AFHAF}) imply that
$\dim\FH_*(\A_\theta)=\infty$ and consequently
the Morse function $\A_\theta$ has infinitely many critical points. 
Now Proposition \ref{prop:crit} implies that there exist infinitely many 
magnetic leaf-wise intersections or a periodic leaf-wise intersection.
But, by Theorem \ref{thm:infgen}, the latter case can be excluded 
for a generic $(H,\theta)\in\Hc\times\mathcal P$ which generates $\varphi$.
Hence there exist infinitely many magnetic leaf-wise intersections.
\end{proof}

\section{On the growth rate of magnetic leaf-wise intersections}
\label{sec:growthmlwip}
In \cite{MMP}, Macarini-Merry-Paternain prove 
the exponential growth rate of leaf-wise intersections
with respect to the period when $\widetilde\pi_1(N)$ grows exponentially.
Recall that $\widetilde\pi_1(N)$ is the set of conjugacy classes of $\pi_1(N)$.

\subsection{Symplectically hyperbolic manifolds}
In this section, we investigate the examples 
and the candidates for the above topological assumption.
\begin{Def}\label{def:symphyp}
Let $(N,\om_N)$ be a closed symplectic manifold of dimension $2n$.
If the symplectic form $\om_N$ is $\widetilde d$-bounded,
then $(N,\om_N)$ is called {\em symplectically hyperbolic}.
\end{Def}

\begin{Prop}[K\c{e}dra \cite{Ked09}]\label{prop:symphypexp}
Let $(N,\om_N)$ be a symplectically hyperbolic manifold 
then $\pi_1(N)$ grows exponentially.
\end{Prop}

As mentioned above, we are interested in the growth rate of
$\widetilde\pi_1(N)$.
It is known that 
$\widetilde\pi_1(N)$ has exponential growth rate
when $N$ is a 2-dimensional symplectically hyperbolic manifold.
But we don't know the growth rate of $\widetilde\pi_1(N)$ 
for any higher dimensional symplectically hyperbolic manifold.

\subsection{Perturbed $\F$-Rabinowitz action functional}\label{sec:fRabact}
In order to show the exponential growth rate of leaf-wise intersection points,
Macarini-Merry-Paternain used the $\F$-Rabinowitz action functional $\A^f:\L\times\R\to\R$ as follows
\bean
\A^f(u,\eta)=\A^{F,f}_H(u,\eta):=\int_0^1 u^*\lambda-f(\eta)\int_0^1F(t,u)dt-\int_0^1H(t,u)dt.
\eea
The above new ingredient $f\in C^\infty(\R,\R)$ needs to satisfy the following properties:
\begin{enumerate}
\item $f$ is a smooth strictly positive, strictly increasing function.
\item $\lim_{\eta\to-\infty}f(\eta)=0$ and $f'$ satisfies $0<f'(\eta)\leq1$ for all $\eta\in\R$.
\end{enumerate}
The additional data $f(\eta)$ is crucial to the construction of
continuation maps between a concentric family of fiberwise starshaped hypersurfaces, see \cite[Section 4.2]{MMP}.
We denote by $\F$ the set of such $f\in C^\infty(\R,\R)$ satisfying the above conditions.

If we additionally consider the magnetic perturbation, then the action functional becomes
\bean
\A^f_\theta(u,\eta)=\A^{F,f}_{H,\theta}:=\A^f(u,\eta)+\Bc_\theta(u).
\eea
One can check that a critical point of $\A^f_\theta$ satisfies 
\beq
\left.
\begin{array}{cc}
\frac{d}{dt}u=f(\eta)X_F(t,u)+X_H^\sigma(t,u) \\
f'(\eta)\int_0^1F(t,u)dt=0.
\end{array}
\right\}
\eeq
Since $f'(\eta)>0$ for all $\eta\in\R$, it is equivalent to
\beq\label{eqn:fcriteqn}
\left.
\begin{array}{cc}
\frac{d}{dt}u=f(\eta)X_F(t,u)+X_H^\sigma(t,u) \\
\int_0^1F(t,u)dt=0.
\end{array}
\right\}
\eeq
Given $-\infty\leq a\leq b\leq\infty$, we adopt the following notations:
\bean
\Crit(\A^f_\theta)&:=\{w=(u,\eta)\in\L\times\R\,:\,(u,\eta)\text{ satisfies (\ref{eqn:fcriteqn})}\};\\
\Crit^{(a,b)}(\A^f_\theta)&:=\{(u,\eta)\in\Crit(\A^f_\theta)\,:\,\A^f_\theta(u,\eta)\in(a,b)\}.
\eea
A (magnetic) leaf-wise intersection point is called {\em positive} or {\em negative} 
if $\eta$ in (\ref{eqn:lwip}) is positive or negative respectively.
Since $f\in\F$ is a positive function, we only consider {\em positive} (magnetic) leaf-wise intersection 
points. It would be convenient if $f(\eta)=\eta$ on the action window $(a,b)\subset\R^+$ we work with. 

\begin{Def}\label{def:F(a)}
Given $a>0$,
\bean
\F(a):=\{f\in\F\,:\,f(\eta)=\eta,\ \forall\eta\in[a,\infty)\}.
\eea
\end{Def}

For notational convenience, let us denote by
\bean
\LW(a,b)=\LW_{\Sigma,\varphi}(a,b):=\{x\in T^*N\,:\,\phi_\eta^\Sigma(\varphi(x))=x,\text{ for some } \eta\in(a,b)\}
\eea
and recall that
\bean
c(H,\theta)=\sup_{(t,u)\in \R/\Z\times\L}
\left|\int_0^1(\widetilde\lambda+\widetilde\tau^*\theta_t)(\widetilde u(t))[\widetilde X_H^\sigma(t,\widetilde u)]-H(t,u(t))dt\right|.
\eea

\begin{Prop}\label{prop:critmlwip}
Given $a>0$, choose $f\in\F(a)$.
Then there is a map
\bean
&\ev:\Crit^{(a+c,b-c)}(\A^f_\theta)\to\LW(a,b)\\
&\ev(u,\eta)= u(\frac{1}{2}).
\eea
Appendix \ref{sec:noper} guarantees that there is no periodic magnetic leaf-wise intersection point
for generic $\varphi$. 
In this case, $\ev$ is injective and we then obtain the following estimate
\bean
\#\LW(a,b)\geq\#\Crit^{(a+c,b-c)}(\A^f_\theta).
\eea
Here $c=c(H,\theta)$.
\end{Prop}

\begin{proof}
Let $(u,\eta)$ be a critical point of $\A^f_\theta$, by the argument in Proposition \ref{prop:crit},
then $u(\frac{1}{2})$ is a magnetic leaf-wise intersection point and 
its action value becomes
\bea\label{eqn:Afmactionest}
\A^f_\theta(u,\eta)&=\int_0^1(\widetilde\lambda+\widetilde\tau^*\theta_t)
\big(f(\eta)\widetilde X_F(t,\widetilde u)+\widetilde X^\sigma_H(t,\widetilde u)\big)
-\int_0^1 H(t,u)dt \\
&=f(\eta)\underbrace{\int_0^1(\widetilde\lambda+\widetilde\tau^*\theta_t)[\widetilde X_F(t,\widetilde u)]dt}_{=:\diamondsuit}
+\int_0^1(\widetilde\lambda+\widetilde\tau^*\theta_t)[\widetilde X^\sigma_H(t,\widetilde u)]dt
-\int_0^1 H(t,u)dt \\
&=f(\eta)+\int_0^1(\widetilde\lambda+\widetilde\tau^*\theta_t)\big(\widetilde X^\sigma_H(t,\widetilde u)\big)-\int_0^1 H(t,u)dt.
\eea
The third equality in the above equation (\ref{eqn:Afmactionest}) is deduced from the following.
Notice here that $X_F(t,u)=\rho(t)X_{\bar{F}}(u)$ and $\rho(t)$ vanishes on $t\in[\frac{1}{2},1]$ 
while $\theta_t=0$ for $t\in[0,\frac{1}{2}]$.
\bean
\diamondsuit
&=\int_0^1(\widetilde\lambda+\widetilde\tau^*\theta_t)[\rho(t)\widetilde X_{\bar{F}}(\widetilde u)]dt
=\int_0^1\rho(t)\widetilde\lambda(\widetilde X_{\bar{F}}(\widetilde u))dt\\
&=\int_0^1\rho(t)\lambda(X_{\bar{F}}(u))dt
=\int_0^1\rho(t)\lambda(R(u))dt\\
&=\int_0^1\rho(t)dt=1.
\eea 
Thus we obtain
\bea\label{eqn:actionetaes}
|\A^f_\theta(u,\eta)-f(\eta)|\leq c(H,\theta).
\eea
Suppose $\A^f_\theta(u,\eta)\in(a+c(H,\theta),b-c(H,\theta))$ then
\bean
a<f(\eta)<b.
\eea
Since $f\in\F(a)$, we conclude that
\bean
a<\eta<b.
\eea

\end{proof}

For a given almost complex structure $J\in\J_{\om}$, let $\nabla_{J}\A^f_\theta$ be the gradient of $\A^f_\theta$
with respect to the metric $\g_J(\cdot,\cdot)$ in (\ref{eqn:metricgJ}). One can check that
\bean
\nabla_J\A^f_\theta(u,\eta)=
\left(
\begin{array}{cc}
-J(t,u)\left(\frac{d}{dt} u-f(\eta)X_F(t,u)-X_H^\sigma(t,u)\right) \\
-f'(\eta)\int_0^1F(t,u)dt
\end{array}
\right).
\eea

\begin{Def}
{\em A positive gradient flow line} of $\A^f_\theta$ 
with respect to an $\R/\Z$-parametrized almost complex structure $J(t)\in\J_{\om_\sigma}$ is a map
$w:\R\to\L\times\R$ which solves
\bean
\frac{d}{ds}w-\nabla_J\A^f_\theta=0.
\eea
The above map is interpreted as $w=(u,\eta)$ 
where $u:\R\times\\R/\Z\to T^*N\times\R$, $\eta:\R\to\R$ such that
\bea\label{eqn:gradeqnfm}
\left.
\begin{array}{cc}
\p_su+J(t,u)\left(\frac{d}{dt}u-X_H^\sigma(t,u)-f(\eta) X_F(t,u)\right)=0 \\
\frac{d}{ds}\eta+f'(\eta)\int_0^1F(t,u)dt=0.
\end{array}
\right\}
\eea
\end{Def}

\subsection{Floer homology for $\A^f_\theta$}\label{sec:fRabFlo}

Let us first assume that the 
perturbed $\F$-Rabinowitz action functional 
$\A^f_\theta:\L\times\R\to\R$ is Morse in the sense of Corollary \ref{cor:afmorse}.
In order to define the Floer homology for $\A^f_\theta$, 
we need to show that the Lagrange multiplier 
$\eta$ is uniformly bounded.
We follow the same strategy as in the $\A_\theta$-case with minor modifications.

\begin{Lemma}\label{lem:f3}
There exist $\epsilon,c'>0$ such that if $(u,\eta)\in\L\times\R$ satisfies
$\|\nabla_J\A^f_\theta(u,\eta)\|_J\leq\epsilon f'(\eta)$ then 
\bea\label{eqn:f(eta)esAA'}
\frac{2}{3}\left(\A^f_\theta(u,\eta)-c'\|\nabla_J\A^f_\theta(u,\eta)\|_J-c\right)
\leq f(\eta) \leq
2\left(\A^f_\theta(u,\eta)+c'\|\nabla_J\A^f_\theta(u,\eta)\|_J+c\right).
\eea
Here $c=c(H,\theta)$ as in Definition \ref{def:c(m)}.
\end{Lemma}

\begin{proof}
The proof consists of 2 steps.

\vspace{2mm}

{\bf Step 1} : {\em There exist constants $\delta,c'>0$ such that if $u\in\L$ satisfies
\bean
u(t)\in U_\delta:=F^{-1}(-\delta,\delta), \quad\forall t\in[0,\frac{1}{2}]
\eea
then (\ref{eqn:f(eta)esAA'}) holds.}

\vspace{2mm}

There exist $\delta>0$ such that
\bean
\frac{1}{2}+\delta\leq\lambda(X_F(p))\leq\frac{3}{2}-\delta,\quad\forall\, p\in U_\delta.
\eea
Now we compute
\bean
\A_\theta^f(u,\eta)
=&\int_0^1u^*\lambda-\int_0^1H(t,u(t))dt-f(\eta)\int_0^1F(t,u(t))dt+\Bc_\theta(u(t)) \\
=&\int_0^1\widetilde u^*(\widetilde\lambda+\widetilde\tau^*\theta_t)-\int_0^1H(t,u(t))dt-f(\eta)\int_0^1F(t,u(t))dt \\
=&\int_0^1(\widetilde\lambda+\widetilde\tau^*\theta_t)(\widetilde u(t))[\frac{d}{dt}\widetilde u-f(\eta)\widetilde X_F(t,\widetilde u)-\widetilde X_H^\sigma(t,\widetilde u)]dt \\
&+f(\eta)\int_0^1\underbrace{\lambda(u(t))[X_F(t,u)]}_{\geq\frac{1}{2}+\delta}-\underbrace{F(t,u(t))}_{\leq\delta}dt \\
&+\int_0^1(\widetilde\lambda+\widetilde\tau^*\theta_t)(\widetilde u(t))[\widetilde X_H^\sigma(t,\widetilde u)]-H(t,u(t))dt \\
\geq&\left(\frac{1}{2}+\delta-\delta\right)f(\eta)-c' \|\frac{d}{dt}u-X_H^{\sigma(s)}(t,u)-f(\eta) X_F(t,u)\|_1-c(H,\theta) \\
\geq&\frac{1}{2}|f(\eta)|-c' \|\frac{d}{dt}u-X_H^{\sigma(s)}(t,u)-f(\eta) X_F(t,u)\|_2-c(H,\theta) \\
\geq&\frac{1}{2}|f(\eta)|-c' \|\nabla_J\A_\theta(u,f(\eta))\|_J-c(H,\theta), \\
\eea
where $c'=c'(\theta,\delta):=\|(\widetilde\lambda+\widetilde\tau^*\theta)|_{\widetilde U_\delta}\|_\infty$. 
In a similar way, we get the following estimate
\bean
\A_\theta^f(u,\eta)
=&\int_0^1u^*\lambda-\int_0^1H(t,u(t))dt-f(\eta)\int_0^1F(t,u(t))dt+\Bc_\theta(u(t)) \\
=&\int_0^1\widetilde u^*(\widetilde\lambda+\widetilde\tau^*\theta_t)-\int_0^1H(t,u(t))dt-f(\eta)\int_0^1F(t,u(t))dt \\
=&\int_0^1(\widetilde\lambda+\widetilde\tau^*\theta_t)(\widetilde u(t))[\frac{d}{dt}\widetilde u-f(\eta)\widetilde X_F(t,\widetilde u)-\widetilde X_H^\sigma(t,\widetilde u)]dt \\
&+f(\eta)\int_0^1\underbrace{\lambda(u(t))[X_F(t,u)]}_{\leq\frac{3}{2}-\delta}-\underbrace{F(t,u(t))}_{\geq-\delta}dt \\
&+\int_0^1(\widetilde\lambda+\widetilde\tau^*\theta_t)(\widetilde u(t))[\widetilde X_H^\sigma(t,\widetilde u)]-H(t,u(t))dt \\
\leq&\left(\frac{3}{2}-\delta+\delta\right)f(\eta)+c'\|\frac{d}{dt}u-X_H^{\sigma(s)}(t,u)-f(\eta) X_F(t,u)\|_1+c(H,\theta) \\
\leq&\frac{3}{2}|f(\eta)|+c'\|\frac{d}{dt}u-X_H^{\sigma(s)}(t,u)-f(\eta) X_F(t,u)\|_2+c(H,\theta) \\
\leq&\frac{3}{2}|f(\eta)|+c'\|\nabla_J\A_\theta(u,f(\eta))\|_J+c(H,\theta). \\
\eea
The above two estimates prove Step 1.

\vspace{2mm}

{\bf Step 2} : {\em For any $\delta>0$ there exist $\epsilon>0$ such that if $(u,\eta)\in\L\times\R$
\bean
\|\nabla_J\A^f_\theta(u,\eta)\|_J\leq\epsilon f'(\eta)
\eea
then $u(t)\in U_\delta$ for all $t\in[0,\frac{1}{2}]$.}

\vspace{2mm}

By a similar argument as in Lemma \ref{lem:fund} Step 2, 
if $F(u(t))\geq\frac{\delta}{2}$ for all $t\in[0,\frac{1}{2}]$ then
\bean
\|\nabla_J\A^f_\theta(u,\eta)\|_J\geq\left|f'(\eta)\int_0^1F(t,u(t))dt\right|\geq f'(\eta)\frac{\delta}{2}.
\eea
Now, if there exist $t_1,t_2$ in $[0,\frac{1}{2}]$ with $F(u(t_1))\leq\frac{\delta}{2}$ and $F(u(t_2))\geq\delta$ then
\bean
\|\nabla_J\A^f_\theta(u,\eta)\|_J\geq\frac{\delta}{2\|\nabla F\|_\infty}.
\eea
If we set
\bean
\epsilon=\epsilon(\delta,F):=\min\left\{\frac{\delta}{2},\frac{\delta}{2\|\nabla F\|_\infty}\right\}
\eea
and use the fact that $f'(\eta)\leq1$ for all $\eta\in\R$ then this proves Step 2.

By combining Step 1 and Step 2, we immediately prove the lemma.
\end{proof}

We need further preliminaries.
Now we consider a certain class of $f\in\F(a)$ with the following condition.

\begin{Def}\label{def:F(a,r)}
Given $a,r>0$,
\bea\label{eqn:fcondi2}
\F(a,r)&:=\{f\in\F(a)\,:\,\exists A>0\text{ such that }Af'(-A)>r\};\\
\widetilde\F(a)&:=\bigcap_{r>0}\F(a,r).
\eea
\end{Def}

\begin{Rmk}
Given $a>0$, the set $\bigcap_{r>0}\F(a,r)$ is non-empty and path-connected.
An explicit construction of $f\in\bigcap_{r>0}\F(a,r)$ exists.
There also exists a homotopy between two different $f_0,f_1\in\F(a,r)$.
All these things are explained in \cite[Remark 3.24, Lemma 3.25]{MMP}.
\end{Rmk}

\begin{Prop}\label{prop:f_etaes}
Fix $F\in\D(\Sigma)$ and an action window $(a,b)$ such that $0<a<b<\infty$.
Let $c',\epsilon>0$ be the constants from Lemma \ref{lem:f3}.
Choose $f\in\F(\frac{a}{6},\frac{b-a}{\min\{\epsilon,a/4c'\}})$ 
and a generic pair $(H,\theta)$ such that $c(H,\theta)\leq\frac{a}{2}$.
Let $w_\pm\in\Crit^{(a,b)}(\A^f_\theta)$ and 
$w=(u,\eta)$ be a gradient flow line of $\A^f_\theta$ with $\lim_{s\to\pm\infty}w(s)=w_\pm$.
Then there exists a constant $\kappa=\kappa(a,b)$ satisfying $\|\eta\|_\infty\leq\kappa$.
\end{Prop}

\begin{proof}
For convenience, set
\bean
\epsilon_1:=\min\left\{\epsilon,\frac{a}{4c'}\right\}.
\eea 
First define a function $\nu_w:\R\to[0,\infty)$ for a given gradient flow line $w=(u,\eta)$ by
\bean
\nu_w(l):=\inf\{\nu\geq0\,:\,\|\nabla_J\A^f_\theta(w(l+\nu))\|_J\leq\epsilon_1 f'(\eta(l+\nu))\}.
\eea
Since $\lim_{s\to\infty}f'(\eta(s))=1$ and $\lim_{s\to\infty}\|\nabla_J\A^f_\theta((u,\eta)(s))\|_J=0$, $\nu_w$ is well-defined.
We get the following estimate
\bea\label{eqn:nuesb-a}
b-a
\geq&\lim_{s\to\infty}\A^f_\theta(w(s))-\lim_{s\to-\infty}\A^f_\theta(w(s)) \\
=&\intinf\|\nabla_J\A^f_\theta(w(s))\|^2_Jds \\
\geq&\int_l^{l+\nu_w(l)}\epsilon_1^2 f'(\eta(s))^2ds \\
\geq&\nu_w(l)\epsilon_1^2 i_w(l)^2,
\eea
where $i_w(l):=\inf_{l\leq s\leq l+\nu_w(s)}f'(\eta(s))$.
Hence we obtain
\bean
\nu_w(l)\leq\frac{b-a}{\epsilon_1^2 i_w(l)^2}.
\eea
Now observe that
\bea\label{eqn:dotetaesb-a}
\left|\int_l^{l+\nu_w(l)}\dot\eta(s)ds\right|
\leq&\int_l^{l+\nu_w(l)}\left|\dot\eta(s)\right|ds \\
\leq&\left(\nu_w(l)\int_l^{l+\nu_w(l)}\left|\dot\eta(s)\right|^2ds\right)^{1/2} \\
\leq&\left(\nu_w(l)\int_l^{l+\nu_w(l)}\|\nabla_J\A^f_\theta(w(s))\|_J^2ds\right)^{1/2} \\
\leq&(\nu_w(l)\ E(w))^{1/2} \\
\leq&\frac{b-a}{\epsilon_1\ i_w(l)}.
\eea
By Lemma \ref{lem:f3}, we get the following estimate for any $l\in\R$
\bean
f[\eta(l+\nu_w(l))]
\geq&\frac{2}{3}\bigg(\A^f_\theta[w(l+\nu_w(l))]-c'\|\nabla_J\A^f_\theta(u,\eta)\|_J-\underbrace{c(H,\theta)}_{\leq\frac{a}{2}}\bigg) \\
\geq&\frac{2}{3}(a-\underbrace{c'\epsilon_1}_{\leq\frac{a}{4}} \underbrace{f'[\eta(l+\nu_w(l))]}_{\leq 1}-\frac{a}{2}) \\
\geq&\frac{a}{6}.
\eea
Since $f\in\F(\frac{a}{6})$, we get
\bean
\eta(l+\nu_w(l))\geq\frac{a}{6},
\eea
and hence
\bean
\eta(l)\geq&\eta(l+\nu_w(l))-\left|\int_l^{l+\nu_w(l)}\dot\eta(s)ds\right| \\
\geq&\frac{a}{6}-\frac{b-a}{\epsilon_1 i_w(l)} \\
>&-\frac{b-a}{\epsilon_1 i_w(l)}.
\eea
This implies
\bean
f'(\eta(l))\eta(l)\geq i_w(l)\eta(l)\geq-\frac{b-a}{\epsilon_1}.
\eea
Now suppose that there exists $l_0\in\R$ such that $\eta(l_0)<-A$
then there must be $l_1\in\R$ with $\eta(l_1)=-A$.
This induces the following contradiction by the choice of $f\in\F(\frac{a}{6},\frac{b-a}{\epsilon_1})$ 
with (\ref{eqn:fcondi2}),
\bean
-\frac{b-a}{\epsilon_1}>-f'(-A)A=f'(\eta(l_1))\eta(l_1)>-\frac{b-a}{\epsilon_1}.
\eea
So, we conclude that $\eta(l)>-A$ for all $l\in\R$.

Now consider the upper bound. Start with a new function $\widetilde\nu_w:\R\to[0,\infty)$ by
\bean
\widetilde\nu_w(l):=\inf\{\nu\geq0:\|\nabla_J\A^f_\theta(w(l+\nu))\|_J\leq\epsilon_1 f'(-A)\}.
\eea
By a similar argument as in (\ref{eqn:nuesb-a}) and (\ref{eqn:dotetaesb-a}), we see that
\bean 
\widetilde\nu_w(l)\leq\frac{b-a}{\epsilon_1^2\ f'(-A)^2}
\eea
and
\bea\label{eqn:etaesA}
|\eta(l)-\eta(l+\widetilde\nu_w(l))|<\frac{b-a}{\epsilon_1\; f'(-A)}<A
\eea
where the last inequality comes from (\ref{eqn:fcondi2}) again.
By Lemma \ref{lem:f3}, we get
\bean
f[\eta(l+\widetilde\nu_w(l))]
&\leq 2\left(\A^f_\theta[w(l+\widetilde\nu_w(l))]+c'\|\nabla_J\A^f_\theta[w(l+\widetilde\nu_w(l))]\|_J+c(H,\theta)\right) \\
&\leq 2(b+\underbrace{c'\epsilon_1}_{\leq\frac{a}{4}} \underbrace{f'(-A)}_{\leq 1}+\frac{a}{2}) \\
&< 2a+2b.
\eea
This implies that $\eta(l+\widetilde\nu_w(l))<2a+2b$ and by (\ref{eqn:etaesA})
\bean
\eta(l)<2a+2b+A.
\eea
Thus we conclude that
\bean
\|\eta\|_\infty<\kappa:=2a+2b+A.
\eea
\end{proof}

For simplicity, let us denote by 
\bean
A(\A^f_\theta):=\{\A^f_\theta(w)\,:\,w\in\Crit(\A^f_\theta)\}.
\eea

\begin{Thm}
Fix $F\in\D(\Sigma)$ and $f\in\widetilde\F(\frac{1}{6})$, see Definition \ref{def:F(a,r)}.
Choose a generic pair $(H,\theta)$. 
If $\max\{1,2\,c(H,\theta)\}<a<b\leq\infty$ and $a,b\notin A(\A^f_\theta)$, 
then $\FH^{(a,b)}(\A^f_\theta)$ is well-defined.
\end{Thm}

\begin{proof}
The construction of $\FH^{(a,b)}(\A^f_\theta)$ is the same as in the $\A_\theta$-case.
For $w=(u,\eta)\in\Crit^{(a,b)}(\A^f_\theta)$,
we define the index $\mu(w):=\CZ(u)$.
Let us denote by
\bean
\Crit_k^{(a,b)}(\A^f_\theta)&:=\{w\in\Crit^{(a,b)}(\A^f_m)\,:\,\mu(w)=k\};\\
\FC_k^{(a,b)}(\A^f_\theta)&:=\Crit_k^{(a,b)}(\A^f_\theta)\otimes\Z_2.
\eea
For a generic almost complex structure $J(t)\in\J_\sigma$ 
and given $w_\pm\in\Crit^{(a,b)}(\A^f_\theta)$, we define
\bean
\widehat\M(w_-,w_+)&:=\{w(s)\,:\,w\text{ satisfies }(\ref{eqn:gradeqnfm}),\ \lim_{s\to\pm\infty}w(s)=w_\pm\};\\
\M(w_-,w_+)&:=\widehat\M(w_-,w_+)/\R.
\eea
The above $\R$-action is given by translating the $s$-coordinate.
Suppose further that the almost complex structure $J(t)$ is generic,
so that $\M(w_-,w_+)$ is a smooth manifold of dimension
\bean
\dim\M(w_-,w_+)=\mu(w_-)-\mu(w_+)-1.
\eea
The boundary operator $\p_k:\FC_k^{(a,b)}(\A^f_\theta)\to\FC^{(a,b)}_{k-1}(\A^f_\theta)$
is defined by
\bean
\p_k w_-:=\sum_{\mu(w_+)=k-1}\#_2\M(w_-,w_+)w_+,
\eea
where $\#_2$ means $\Z_2$-counting. 
By virtue of Proposition \ref{prop:f_etaes} with Theorem \ref{thm:compofmod}, 
the boundary operator satisfies $\p_{k-1}\circ\p_k=0$.
Then the resulting filtered Floer homology group is
\bean
\FH^{(a,b)}_*(\A^f_\theta)=\H_*(\FC^{(a,b)}_\bullet(\A^f_\theta),\p_*).
\eea
\end{proof}

\subsection{Continuation map between $\FH(\A^f)$ and $\FH(\A^f_\theta)$. }

In this section we construct a continuation homomorphism between 
$\FC(\A^f)$ and $\FC(\A^f_\theta)$
which induces a map on homologies on a suitable action window.
The construction is given by 
counting gradient flow lines of the $s$-dependent action functional
\bean
\A_{\theta(s)}^f(u,\eta):=\A^f(u,\eta)+\gamma(s)\Bc_{\theta}(u).
\eea
Here $\A^f(u,\eta)=\int_0^1u^*\lambda-f(\eta)\int_0^1F(t,u(t))dt-\int_0^1H(t,u(t))dt$ 
and $\sigma(s)$ is defined in (\ref{eqn:m(s)def}).
With the same metric as in (\ref{eqn:m(s)metric}),
the gradient flow line $w=(u,\eta)\in C^\infty(\R\times (\R/\Z),T^*N)\times C^\infty(\R,\R)$ satisfies
\bea\label{eqn:fgradeqn}
\left.
\begin{array}{cc}
\p_su+J(s,t,u)\big(\frac{d}{dt}u-f(\eta) X_F(t,u)-X_H^{\sigma(s)}(t,u)\big)=0 \\
\frac{d}{ds}\eta+f\,'(\eta)\int_0^1F(t,u)dt=0.
\end{array}
\right\}
\eea

In order to construct a continuation map,
we need to check that the energy $\intinf\|\frac{d}{ds}w\|_s^2ds$ 
and the Lagrange multiplier $\eta$ of gradient flow lines $w$
are uniformly bounded.
As in the $\A_\theta$ case, we start with the fundamental lemma.

\begin{Lemma}\label{lem:m(s)f}
There exist $\overline\epsilon,\overline c\,'>0$ such that if $(u,\eta)\in\L\times\R$ satisfies
\bean
\|\nabla_s\A^f_{\theta(s)}(u,\eta)\|_s\leq\overline\epsilon f'(\eta)
\eea
then 
\bean
\frac{2}{3}\left(\A^f_{\theta(s)}(u,\eta)-\overline c\,'\|\nabla_s\A^f_{\theta(s)}(u,\eta)\|_s-\overline c\right)
\leq f(\eta) \leq
2\left(\A^f_{\theta(s)}(u,\eta)+\overline c\,'\|\nabla_s\A^f_{\theta(s)}(u,\eta)\|_s+\overline c\right). 
\eea
Here
\bean
\overline c=\overline c(H,\theta)
:=\sup_{s\in\R}\sup_{(t,u)\in \R/\Z\times\L}
\left|\int_0^1(\widetilde\lambda+\gamma(s)\widetilde\tau^*\theta_t)(\widetilde u(t))[\widetilde X_H^{\sigma(s)}(t,\widetilde u)]-H(t,u(t))dt\right|.
\eea
\end{Lemma}

\begin{proof}
The proof is similar as in Lemma \ref{lem:f3} with $\sigma(s)$
instead of $\sigma$.
So we omit the proof.
With a simple computation, one checks that
\bean
\overline\epsilon=\overline\epsilon(\delta,F)
:=\min\left\{\frac{\delta}{2},\frac{\delta}{2\|\nabla F\|_\infty}\right\}
\eea
and
\bean
\overline c\,'=\overline c\,'(\theta,\delta)
:=\sup_{s\in\R}\|(\widetilde\lambda+\gamma(s)\widetilde\tau^*\theta_t)|_{\widetilde U_\delta}\|_\infty.
\eea
Here $\delta$ is chosen to satisfy
\bean
\frac{1}{2}+\delta\leq\lambda(X_F(p))\leq\frac{3}{2}-\delta,\quad\forall p\in U_\delta.
\eea
\end{proof}

Before state the next proposition, we summarize the notations as follows:
\bea\label{eqn:bunofnot}
C&=\|\theta\|_\infty ;\\
d_\sigma&=d_{H,\sigma}=\sup_{s\in\R}\|X_H^{\sigma(s)}\|_{\infty} ;\\
d_{F}&=\|X_F\|_{\infty};\\
\Delta&=\A^f_{\theta(1)}(w_+)-\A^f_{\theta(0)}(w_-).
\eea

\begin{Prop}\label{prop:f_etaesm(s)}
Fix $F\in\D(\Sigma)$ and an action window $(a,2a)$ with $a\geq2$.
Let $\overline c, \overline c\,',\overline\epsilon>0$ be the constants from Lemma \ref{lem:m(s)f}.
Choose $f\in\F(\frac{a}{6},\frac{2a+1}{\min\{\overline\epsilon,a/8\overline c\,'\}})$ 
and a generic pair $(H,\theta)$ such that $\overline c=\overline c(H,\theta)\leq\frac{a}{2}$.
Let $w$ be a gradient flow line of $\A^f_{\theta(s)}$
with the following asymptotic conditions
\bean
\lim_{s\to-\infty}w(s)=w_-\in\Crit^{(a,2a)}(\A^f_{\theta(0)}),
\quad \lim_{s\to\infty}w(s)=w_+\in\Crit^{(a,2a)}(\A^f_{\theta(1)}).
\eea
If the isoperimetric constant $C=C(\theta)$ satisfies the following conditions:
\bea\label{eqn:CcondiforAf-case}
C&\leq\frac{1}{4};\\
\left(16\,d_F\,C+4\,d_F+\frac{4\|F\|_\infty d_F}{\overline\epsilon^2}\right)C&\leq\frac{1}{2};\\
8(4d_F C^2+d_F C)(2+4C+\frac{\|F\|_\infty}{\overline\epsilon^2})+4C&\leq\frac{1}{16};\\
8C^2+8d_\sigma C^2+2C+2d_\sigma C& \\
+8(4d_F C^2+d_F C)\frac{\|F\|_\infty}{\overline\epsilon^2}(2C+2d_\sigma C)& \\
+8(4d_F C^2+d_F C)\bigg(8C^2+8C^2d_\sigma +2C+2d_\sigma C+\overline c\,'\overline\epsilon+\overline c
+\bigg)&\leq\frac{1}{8};\\ 
4C+4d_\sigma C+8\frac{\|F\|_\infty}{\overline\epsilon^2}(2\Delta+4C+4d_\sigma C)C& \\
+8(4a+8C\Delta+16C^2+16C^2d_\sigma +4C+4d_\sigma C+2\overline c\,'\overline\epsilon+2\overline c)C&\leq1,
\eea 
then the $L^\infty$-norm of $\eta$ is 
uniformly bounded in terms of a constant which only depends on $w_-$, $w_+$.
\end{Prop}

\begin{proof}
The proof consists of 4 steps.

\vspace{2mm}

{\bf Step 1} : {\em The energy is bounded by $\|f(\eta)\|_\infty$.}

\vspace{2mm}

By a similar argument as in Proposition \ref{prop:conti} Step 1, we obtain
\bea\label{eqn:dotAesf}
\intinf\left|\dot\A^f_{\theta(s)}(w(s))\right|ds\leq2CE(w)+2C+2d_\sigma C+2\|f(\eta)\|_\infty d_{F}C \\
\eea
and
\bea\label{eqn:Eesf}
E(w)\leq 2\Delta+4C+4d_\sigma C+4\|f(\eta)\|_{\infty} d_{F}C,
\eea
under the assumption that the isoperimetric constant $C<\frac{1}{4}$.
\bea\label{eqn:Ccondi-1}
.
\eea

\vspace{2mm}

{\bf Step 2} : {\em $\eta(s)$ is uniformly bounded from above}.

\vspace{2mm}

In this step, without loss of generality, we work on the region that $\eta(s)\geq\frac{a}{6}$.
Since $f\in\F(\frac{a}{6})$, $f(\eta(s))=\eta(s)$ and $f'(\eta(s))=1$.
Then Lemma \ref{lem:m(s)f} implies the following:

There exist $\overline\epsilon,\overline c,\overline c\,'>0$ such that if $(u,\eta)\in\L\times\R$ satisfies
\bean
\|\nabla_s\A^f_{\theta(s)}(u,\eta)\|_s\leq\overline\epsilon
=\overline\epsilon \underbrace{f'(\eta)}_{=1}
\eea 
then 
\bea\label{eqn:fesA}
f(\eta) \leq
2\left(\A^f_{\theta(s)}(u,\eta)+\overline c\,'\|\nabla_s\A^f_{\theta(s)}(u,\eta)\|_s+\overline c\right),
\eea
for all $s\in\R$ satisfying $\eta(s)\geq\frac{a}{6}$.
Here $\overline\epsilon,\overline c,\overline c\,'>0$ are same as in Lemma \ref{lem:m(s)f}.

Now define
\bean
\overline\nu_w(l):=\inf\{\overline\nu\geq0:\|\nabla_s\A^f_{\theta(l+\overline\nu)}(w(l+\overline\nu))\|_s<\overline\epsilon\},
\eea
for $l\in\R$ such that $\eta(l)\geq\frac{a}{6}$.
Then by a similar argument as in (\ref{eqn:nuesenergy}), 
we obtain the following estimate
\bea\label{eqn:barnuesE}
\overline\nu_w(l)\leq\frac{E(w)}{\overline\epsilon^2}.
\eea
By the gradient flow equation (\ref{eqn:fgradeqn}) and (\ref{eqn:barnuesE}), we have
\bea\label{eqn:fetaes1}
\left|\int_l^{l+\overline\nu_w(l)}\frac{d}{ds} f(\eta(s))ds\right|
&\leq\Bigg|\int_l^{l+\overline\nu_w(l)} \underbrace{f'(\eta(s))}_{\leq 1}\frac{d}{ds}\eta(s)ds\Bigg| \\
&\leq\Bigg|\int_l^{l+\overline\nu_w(l)} \frac{d}{ds}\eta(s)ds\Bigg| \\
&=\Bigg|\int_l^{l+\overline\nu_w(l)} \underbrace{f'(\eta(s))}_{\leq 1}\underbrace{\int_0^1 F(t,u)dt}_{\leq\|F\|_\infty}\,ds\Bigg| \\
&\leq\|F\|_\infty\overline\nu_w(l) \\
&\leq\|F\|_\infty\frac{E(w)}{\overline\epsilon^2}.\\
\eea
Let us note that the following inequality holds for all $s\in\R$
\bea\label{eqn:a<A<2a}
a-\intinf\left|\dot\A^f_{\theta(s)}(w(s))\right|ds
\leq\A^f_{\theta(s)}(w(s))
\leq 2a+\intinf\left|\dot\A^f_{\theta(s)}(w(s))\right|ds.
\eea
By the definition of $\overline\nu_w(l)$ 
and the above estimates (\ref{eqn:fesA}), (\ref{eqn:a<A<2a}) and (\ref{eqn:dotAesf}) we get
\bea\label{eqn:fetaes2}
f[\eta(l+\overline\nu_w(l))]
\leq&2\bigg(\A^f_{\theta[l+\overline\nu_w(l)]}[w(l+\overline\nu_w(l))]
+\overline c\,'\underbrace{\|\nabla_s\A^f_{\theta[l+\overline\nu_w(l)]}[w(l+\overline\nu_w(l))]\|_s}_{\leq\overline\epsilon}
+\overline c\bigg) \\
\leq&2\left(2a+\intinf|\dot\A^f_{\theta(s)}(w(s))|ds+\overline c\,'\overline\epsilon+\overline c \right) \\
\leq&2\left(2a+2CE(w)+2C+2d_\sigma C+2\|f(\eta)\|_\infty d_{F}C+\overline c\,'\overline\epsilon+\overline c \right). \\
\eea
Now combine (\ref{eqn:fetaes1}) and (\ref{eqn:fetaes2}), we then obtain
\bea\label{eqn:fesfinf}
f(\eta(l))\leq& f[\eta(l+\overline\nu_w(l))]+\left|\int_l^{l+\overline\nu_w(l)}\frac{d}{ds} f(\eta(s))ds\right| \\
\leq& 2\left(2a+2CE(w)+2C+2d_\sigma C+2\|f(\eta)\|_\infty d_{F}C+\overline c\,'\overline\epsilon+\overline c \right)+\|F\|_\infty\frac{E(w)}{\overline\epsilon^2} \\
\leq& \left(16 C d_F+4d_F+\frac{4\|F\|_\infty d_F}{\overline\epsilon^2}\right)C\|f(\eta)\|_\infty \\
&+4a+8C\Delta+16C^2+16C^2d_\sigma +4C+4d_\sigma C+2\overline c\,'\overline\epsilon+2\overline c\\
&+\frac{\|F\|_\infty}{\overline\epsilon^2}(2\Delta+4C+4d_\sigma C),
\eea
where for the last inequality we use (\ref{eqn:Eesf}).
Note that the last line of the above estimate (\ref{eqn:fesfinf}) does not depend on the choice of a gradient flow line $w$
and $l\in\R$.
By the 2nd assumption in (\ref{eqn:CcondiforAf-case}) on the isoperimetric constant $C$ we have
\bean
\left(16\,d_F\,C+4\,d_F+\frac{4\|F\|_\infty d_F}{\overline\epsilon^2}\right)C\leq\frac{1}{2}
\eea
and we conclude 
\bea\label{eqn:unifeta}
\|f(\eta)\|_\infty\leq&2(4a+8C\Delta+16C^2+16C^2d_\sigma +4C+4d_\sigma C+2\overline c\,'\overline\epsilon+2\overline c)\\
&+\frac{2\|F\|_\infty}{\overline\epsilon^2}(2\Delta+4C+4d_\sigma C)=:\overline\kappa.
\eea
Since $f(\eta(s))=\eta(s)$ for $s\geq\frac{a}{6}$,
this implies that $\overline\kappa$ is an uniform upper bound of $\eta(s)$.\\

\vspace{2mm}

{\bf Step 3} : $\intinf|\dot\A^f_{\theta(s)}(w(s))|ds\leq\frac{a}{8}$.

\vspace{2mm}

The above estimates (\ref{eqn:dotAesf}), (\ref{eqn:Eesf}), (\ref{eqn:unifeta}) and $\Delta<a$ imply that 
\bea\label{eqn:dotAfses}
\intinf|\dot\A^f_{\theta(s)}|ds\leq&2CE(w)+2C+2d_\sigma C+2\|f(\eta)\|_\infty d_{F}C \\
\leq&(8d_F C^2+2d_F C)\|f(\eta)\|_\infty+4aC+8C^2+8d_\sigma C^2+2C+2d_\sigma C \\
\leq&8(4d_F C^2+d_F C)(2a+4aC+8C^2+8C^2d_\sigma +2C+2d_\sigma C+\overline c\,'\overline\epsilon+\overline c)\\
	&+8(4d_F C^2+d_F C)\frac{\|F\|_\infty}{\overline\epsilon^2}(a+2C+2d_\sigma C) \\
	&+4aC+8C^2+8d_\sigma C^2+2C+2d_\sigma C \\
=&\bigg(8(4d_F C^2+d_F C)(2+4C+\frac{\|F\|_\infty}{\overline\epsilon^2})+4C\bigg)a \\
	&+8(4d_F C^2+d_F C)\bigg(8C^2+8C^2d_\sigma +2C+2d_\sigma C+\overline c\,'\overline\epsilon+\overline c\bigg) \\
	&+8(4d_F C^2+d_F C)\frac{\|F\|_\infty}{\overline\epsilon^2}(2C+2d_\sigma C) \\
	&+8C^2+8d_\sigma C^2+2C+2d_\sigma C. \\
\eea
Recall the 3rd, 4th condition in (\ref{eqn:CcondiforAf-case}) 
\bean
8(4d_F C^2+d_F C)(2+4C+\frac{\|F\|_\infty}{\overline\epsilon^2})+4C&\leq\frac{1}{16};\\
8C^2+8d_\sigma C^2+2C+2d_\sigma C& \\
+8(4d_F C^2+d_F C)\frac{\|F\|_\infty}{\overline\epsilon^2}(2C+2d_\sigma C)& \\
+8(4d_F C^2+d_F C)\bigg(8C^2+8C^2d_\sigma +2C+2d_\sigma C+\overline c\,'\overline\epsilon+\overline c
+\bigg)&\leq\frac{1}{8}.
\eea 
Then (\ref{eqn:dotAfses}) is simplified as follows
\bean
\intinf|\dot\A^f_{\theta(s)}|ds\leq
\frac{a}{16}+\frac{1}{8}
\leq\frac{a}{8},
\eea
where for the last inequality we use $a\geq2$.
This proves Step 3.

\vspace{2mm}

{\bf Step 4} : {\em $\eta(s)$ is uniformly bounded.}

\vspace{2mm}
First set
\bean
\underline\epsilon:=\min\left\{\overline\epsilon,\frac{a}{8\overline c\,'}\right\},
\eea
and define a function $\underline\nu_w:\R\to\R^{\geq0}$ by
\bean
\underline\nu_w(l):=\inf\{\underline\nu\geq 0 : \|\nabla_s\A^f_{\theta(l+\underline\nu)}(w(l+\underline\nu))\|_s<\underline\epsilon\,f'(\eta(l+\underline\nu))\}.
\eea 
Now set
\bean
\underline i_w(l):=\inf_{l\leq s\leq l+\underline\nu_w(l)}f'(\eta(s)).
\eea
By similar arguments as in (\ref{eqn:nuesb-a}) and (\ref{eqn:dotetaesb-a}), 
we obtain the following estimates
\bean
\underline\nu_w(l)\leq\frac{E(w)}{\underline\epsilon^2\,\underline i_w(l)^2}
\eea
and
\bean
|\eta(l)-\eta(l+\underline\nu_w(l))|\leq\frac{E(w)}{\underline\epsilon\,\underline i_w(l)}.
\eea
Lemma \ref{lem:m(s)f} implies that for any $l\in\R$
\bea\label{eqn:fetaesA,a}
f[\eta(l+\underline\nu_w(l))]
\geq&\frac{2}{3}\left(\A^f_{\theta[l+\underline\nu_w(l)]}[w(l+\underline\nu_w(l))]
-\overline c\,'\|\nabla_s\A^f_{\theta[l+\underline\nu_w(l)]}[w(l+\underline\nu_w(l))]\|_s-\overline c\right) \\
\geq&\frac{2}{3}\bigg(\A^f_{\theta[l+\underline\nu_w(l)]}[w(l+\underline\nu_w(l))]
-\underbrace{\overline c\,'\underline\epsilon}_{\leq\frac{a}{8}} \underbrace{f'[\eta(l+\underline\nu_w(l))]}_{\leq 1}-\frac{a}{2}\bigg) \\
\geq&\frac{2}{3}\left(\A^f_{\theta[l+\underline\nu_w(l)]}[w(l+\underline\nu_w(l))]-\frac{5}{8}a\right).
\eea
Here the 2nd inequality in (\ref{eqn:fetaesA,a}) comes from 
the definition of $\underline\nu_w$ and the assumption $\overline c\leq\frac{a}{2}$,
the 3rd inequality use the definition of $\underline\epsilon$ and $f'\leq1$.
The action estimate (\ref{eqn:a<A<2a}) and Step 3 give us the following estimate
\bea\label{eqn:Afesa}
\A^f_{\theta(s)}(w(s))\geq a-\intinf\left|\dot\A^f_{\theta(s)}(w(s))\right|ds\geq \frac{7}{8}a.
\eea
Let us combine (\ref{eqn:fetaesA,a}), (\ref{eqn:Afesa}) to obtain
\bean
f[\eta(l+\underline\nu_w(l))]
\geq\frac{2}{3}\left(\A^f_{\theta[l+\underline\nu_w(l)]}[w(l+\underline\nu_w(l))]-\frac{5}{8}a\right)
\geq\frac{a}{6}.
\eea
Since $f\in\F(\frac{a}{6})$,
\bean
\eta(l+\underline\nu_w(l))\geq\frac{a}{6}>0,
\eea
and hence
\bean
\eta(l)\geq\frac{a}{6}-\frac{E(w)}{\underline\epsilon\,\underline i_w(l)}
\geq-\frac{E(w)}{\underline\epsilon\,\underline i_w(l)}.
\eea
As a consequence,
\bean
-f'(\eta(l))\eta(l)\leq-\underline i_w(l)\eta(l)
\leq\frac{E(w)}{\underline\epsilon}
\leq\frac{1}{\underline\epsilon}(2\Delta+4C+4d_\sigma C+4\overline\kappa d_F C),
\eea
where the last inequality comes from (\ref{eqn:Eesf}) and 
(\ref{eqn:unifeta}).
By the last condition in (\ref{eqn:CcondiforAf-case}) for the isoperimetric constant $C$ we estimate
\bea\label{eqn:fCcondi2}
(4+4d_\sigma +4\overline\kappa d_F)C\leq 1,
\eea
then
\bean
-f'(\eta(l))\eta(l)\leq\frac{2a+1}{\underline\epsilon},
\eea
here we use again $\Delta<a$.
Since $f\in\F(\frac{a}{6},\frac{2a+1}{\underline\epsilon})$,
there exists $A>0$ such that
\bean
Af'(-A)>\frac{2a+1}{\underline\epsilon}.
\eea
Now suppose that there exists $l_0\in\R$ such that
$\eta(l_0)<-A$ then by continuity 
there exists $l_1\in\R$ such that
$\eta(l_1)=-A$ which leads to a contradiction via condition (\ref{eqn:fcondi2})
\bean
\frac{2a+1}{\underline\epsilon}
< f'(-A)A=-f'(\eta(l_1))\eta(l_1)
<\frac{2a+1}{\underline\epsilon}.
\eea
Thus we conclude that $\eta(l)>-A$ for all $l\in\R$,
and hence
\bean
\|\eta(l)\|_\infty\leq\kappa:=\max\{\overline\kappa,A\}.
\eea
\end{proof}

\begin{Lemma}\label{lem:acvales}
Fix $F\in\D(\Sigma)$ and an action window $(a,2a)$ such that $a\geq2$.
Let $\overline c, \overline c\,',\overline\epsilon>0$ be the constants from Lemma \ref{lem:m(s)f}.
Choose $f\in\F(\frac{a}{6},\frac{2a+1}{\min\{\overline\epsilon,a/8\overline c\,'\}})$ 
and a generic pair $(H,\theta)$ such that $\overline c=\overline c(H,\theta)\leq\frac{a}{2}$.
Let $w$ be a gradient flow line of $\A^f_{\theta(s)}$
with the following asymptotic conditions:
\bean
\lim_{s\to-\infty}w(s)=w_-\in\Crit^{(a,2a)}(\A^f_{\theta(0)}),
\quad \lim_{s\to\infty}w(s)=w_+\in\Crit^{(a,2a)}(\A^f_{\theta(1)}).
\eea
If the isoperimetric constant $C=C(\theta)$ satisfies the following condition
\bea\label{eqn:fCcondi3}
8d_F\left(\frac{\|F\|_\infty}{\overline\epsilon}+4C+1\right)C&\leq\frac{1}{9}; \\
2\bigg(1+d_\sigma +8d_F\frac{\|F\|_\infty}{\overline\epsilon^2}C+8d_F d_\sigma \frac{\|F\|_\infty}{\overline\epsilon^2}C& \\
+32d_F C^2+32d_F d_\sigma C^2+8d_F C+8d_F d_\sigma  C+ 4\overline c\,'\overline\epsilon d_F+4\overline c d_F\bigg)C&\leq\frac{1}{9};
\eea
then
\bean
\A^f_{\theta(1)}(w_+)\geq\frac{9}{10}\A^f_{\theta(0)}(w_-)-\frac{1}{10}.
\eea

\end{Lemma}

\begin{proof}
For notational simplicity, let us denote by
\bean
p=\A^f_{\theta(0)}(w_-),\quad q=\A^f_{\theta(1)}(w_+).
\eea
By Step 2 in Proposition \ref{prop:f_etaesm(s)}, $f(\eta)$ is uniformly bounded as follows,
\bean
\|f(\eta)\|_\infty\leq&2(2q+8C(q-p)+16C^2+16C^2d_\sigma +4C+4d_\sigma C+2\overline c\,'\overline\epsilon+2\overline c)\\
&+\frac{2\|F\|_\infty}{\overline\epsilon^2}(2(q-p)+4C+4d_\sigma C).
\eea
Since $E(w)\geq0$, we obtain the following inequality from (\ref{eqn:Eesf})
\bean
q\geq p-2C-2d_\sigma C-2\|f(\eta)\|_\infty d_F C.
\eea
Now we estimate
\bean
q\geq&p-2C-2d_\sigma C-2\|f(\eta)\|_\infty d_F C\\
\geq&p-2C-2d_\sigma C-4d_F C\bigg(\frac{\|F\|_\infty}{\overline\epsilon^2}(2(q-p)+4C+4d_\sigma C)\\
&+2q+8C(q-p)+16C^2+16C^2d_\sigma +4C+4d_\sigma C+2\overline c\,'\overline\epsilon+2\overline c\bigg)\\
=&p+8d_F\left(\frac{\|F\|_\infty}{\overline\epsilon^2}+4C\right)C\,p
-8d_F\left(\frac{\|F\|_\infty}{\overline\epsilon}+4C+1\right)C\,q
-2\bigg(1+d_\sigma +8d_F\frac{\|F\|_\infty}{\overline\epsilon^2}C\\
&+8d_F d_\sigma \frac{\|F\|_\infty}{\overline\epsilon^2}C
+32d_F C^2+32d_F d_\sigma C^2+8d_F C+8d_F d_\sigma  C+ 4\overline c\,'\overline\epsilon d_F+4\overline c d_F\bigg)C \\
\geq&p+\underbrace{8d_F\left(\frac{\|F\|_\infty}{\overline\epsilon^2}+4C\right)C}_{\geq0}\,p-\frac{1}{9}q-\frac{1}{9} \\
\geq&p-\frac{1}{9}q-\frac{1}{9}.
\eea
Here the 4th inequality we use the assumption (\ref{eqn:fCcondi3}) on $C$.
This proves the assertion.
\end{proof}

For convenience, let us abbreviate
\bea\label{eqn:defofh[p]}
h[p]&:=\frac{9}{10}p-\frac{1}{10}.
\eea

\begin{Lemma}\label{lem:comparison}
Fix $F\in\D(\Sigma)$ and an action window $(a,2a)$ such that $a\geq2$.
Let $\overline c, \overline c\,',\overline\epsilon>0$ be the constants from Lemma \ref{lem:m(s)f}.
Choose $f\in\F(\frac{a}{6},\frac{2a+1}{\min\{\overline\epsilon,a/8\overline c\,'\}})$ 
and a generic pair $(H,\theta)$ such that $\overline c=\overline c(H,\theta)\leq\frac{a}{2}$.
If the isoperimetric constant $C=C(\theta)$ satisfies the conditions 
in Proposition \ref{prop:f_etaesm(s)} and Lemma \ref{lem:acvales}
then there exists a commutative diagram:
\bean
\xymatrix{
\FH^{(h^{-2}[a],2a)}(\A^f_{\theta(0)}) \ar[rr]^{i(h^{-2}[a],h^2[2a])} \ar[rd]_{\widetilde\Phi^\sigma} &&\FH^{(a,h^2[2a])}(\A^f_{\theta(0)}) \\
&\FH^{(h^{-1}[a],h[2a])}(\A^f_{\theta(1)}) \ar[ru]_{\widetilde\Phi_\sigma}
}
\eea
\end{Lemma}

\begin{proof}
Let us first construct $\widetilde\Phi^\sigma$.
Let $w$ be the gradient flow line of $\A^f_{\theta(s)}$ 
satisfying the limit conditions:
\bean
\lim_{s\to-\infty}w(s)=w_-\in\Crit^{(h^{-2}[a],2a)}_k(\A^f_{\theta(0)}),\qquad
\lim_{s\to\infty}w(s)=w_+\in\Crit^{(h^{-1}[a],h[2a])}_k(\A^f_{\theta(1)}).
\eea
Let $\M(w_-,w_+)$ be the moduli space of such gradient flow lines.
In order to compactify the moduli space $\M(w_-,w_+)$,
by similar arguments in Theorem \ref{thm:continuation},
it suffices to bound the energy $E(w)=\intinf\|\frac{d}{ds}w(s)\|^2_sds$ 
and the Lagrange multiplier $\eta$ in terms of $w_-,w_+$.
By the assumption on the isoperimetric constant $C$,
we can use the argument of Proposition \ref{prop:f_etaesm(s)} and Lemma \ref{lem:acvales}.
Especially (\ref{eqn:Eesf}), (\ref{eqn:unifeta}) give us the following uniform energy bound
\bean
E(w)
\leq& 2\A^f_{\theta(1)}(w_+)-2\A^f_{\theta(0)}(w_-)+4\,C+4d_\sigma C+4\|f(\eta)\|_{\infty} d_{F}\,C \\
\leq& 2\A^f_{\theta(1)}(w_+)-2\A^f_{\theta(0)}(w_-)+4\,C+4d_\sigma C+4\,\overline\kappa(w_-,w_+)\,d_{F}\,C
\eea
and Proposition \ref{prop:f_etaesm(s)} enables us to conclude that
the Lagrange multiplier $\eta$ is also uniformly bounded.

If $\mu(w_-)=\mu(w_+)$, then $\M(w_-,w_+)$ is discrete
for a generic almost complex structure $J(s,t)\in\J_{\sigma(s)}$.
By virtue of Lemma \ref{lem:acvales},
we now define a map
\bean
\Phi^\sigma_*:\FC_*^{(h^{-2}[a],2a)}(\A^f_{\theta(0)})\to \FC_*^{(h^{-1}[a],h[2a])}(\A^f_{\theta(1)})
\eea
given by
\bean
\Phi^\sigma_*(w_-)=\sum_{\mu(w_+)=\mu(w_-)}\#_2\M(w_-,w_+)w_+,
\eea
where $\#_2$ means $\Z_2$ counting.
Since the continuation map $\Phi^\sigma$ commutes with the boundary operators,
this induces the following homomorphism on homologies as follows
\bean
\widetilde\Phi^\sigma:\FH^{(h^{-2}[a],2a)}(\A^f_{\theta(0)})\to \FH^{(h^{-1}[a],h[2a])}(\A^f_{\theta(1)}).
\eea

Now we consider the inverse homotopy of $\A^f_{\theta(s)}$.
By modifying the above construction, we obtain 
\bean
\widetilde\Phi_\sigma:\FH^{(h^{-1}[a],h[2a])}(\A^f_{\theta(1)})\to \FH^{(a,h^2[2a])}(\A^f_{\theta(0)}).
\eea
By a homotopy-of homotopies argument, 
we conclude that $\widetilde\Phi_\sigma\circ\widetilde\Phi^\sigma$ 
is the identity map on $\FH^{(h^{-2}[a],h^2[2a])}(\A^f_{\theta(0)})$.
This proves the lemma.
\end{proof}

\begin{Lemma}\label{lem:dimcompa0}
Fix $F\in\D(\Sigma)$ and 
$f\in\widetilde\F(\frac{1}{6})$, 
see Definition \ref{def:F(a,r)}.
Choose a generic pair $(H,\theta)$ such that 
the isoperimetric constant $C=C(\theta)$ satisfies the conditions 
in Proposition \ref{prop:f_etaesm(s)} and Lemma \ref{lem:acvales}.
If $\max\{2,2\overline c(H,\theta)\}<a<T<\infty$ then
\bea\label{eqn:compa0con}
\dim\FH^{(h^{-1}[a],h[T])}(\A^f_\theta)
\geq\frac{1}{4}\dim\FH^{(h^{-2}[a],h^2[T])}(\A^f)
\eea
holds for generic $a,T$, see (\ref{eqn:defofh[p]}) for $h[p]$.
\end{Lemma}

\begin{proof}
By the commutative diagram in Lemma \ref{lem:comparison},
we obtain the following dimension estimate
\bea\label{eqn:dimcomp0}
\dim\FH^{(h^{-1}[a],h[2a])}(\A^f_\theta)
&\geq\rank\big(i(h^{-2}[a],h^2[2a])\big) \\
&\geq\dim\FH^{(h^{-2}[a],h^2[2a])}(\A^f).
\eea 
Actually if we choose $b\in\R$ such that $a<b<2a$ then
\bean
\dim\FH^{(h^{-1}[a],h[b])}(\A^f_\theta)
&\geq\dim\FH^{(h^{-2}[a],h^2[b])}(\A^f)
\eea 
holds under the generic condition 
$h[b]\notin A(\A^f_\theta)$ and $h^2[b]\notin A(\A^f)$.

Now we construct a sequence $\{a_i\}_{i=1}^\infty$ such that
the following holds:
\begin{itemize}
\item $a_1\geq\max\{2,2\overline c(H,\theta)\}$;
\item $a_{i+1}=h^2[2a_i]$;
\item $h^{-1}[a_i]\notin A(\A^f_\theta),\ \forall\,i\in\N$;
\item $h^{-2}[a_i],h^2[2a_i]\notin A(\A^f),\ \forall\,i\in\N.$
\end{itemize}
Note that $a_1$ determines the sequence 
and obviously $\{a_i\}$ is strictly increasing.
The 3rd and 4th conditions are guaranteed for a generic $a_1$.
Let $\mathfrak a$ be the set of sequences satisfying the
above conditions.

In order to compare $\dim\FH^{(h^{-1}[a],h[T])}(\A^f_\theta)$
and $\dim\FH^{(h^{-2}[a],h^2[T])}(\A^f)$,
we use (\ref{eqn:dimcomp0}) inductively.
Choose $\{a_i\}\in\mathfrak a$ then the following holds:
\bea\label{eqn:dimcomp1}
\dim\FH^{(h^{-1}(a_1),h(2a_k))}(\A^f_\theta)
&=\sum_{i=1}^{k}\dim\FH^{(h^{-1}[a_i],h[2a_i])}(\A^f_\theta) \\
&\geq\sum_{i=1}^{k}\dim\FH^{(h^{-2}[a_i],h^2[2a_i])}(\A^f).
\eea
But there exist missing action intervals for $\A^f$ 
in the last term of (\ref{eqn:dimcomp1}).
To cover the missing intervals,
we first observe that if $a\geq2$ 
then the length of the action intervals for $\A^f_\theta$ and $\A^f$
\bean
h[2a]-h^{-1}[a],\quad h^2[2a]-h^{-2}[a]
\eea
are positive and increasing functions with respect to $a$.
By a simple computation, one can check that its ratio satisfies
\bean
\frac{h[2a]-h^{-1}[a]}{h^2[2a]-h^{-2}[a]}\leq 4
\eea
for all $a\geq 2$.
This implies that there exist 4 sequences 
$\{a^1_i\},\{a^2_i\},\{a^3_i\},\{a^4_i\}\in\mathfrak a$
such that
$a^1_1<a^2_1<a^3_1<a^4_1<a^1_2$ and
\bean
\big(h^{-2}[a^1_k],h^2[2a^1_k]\big)\cup
\bigcup_{i=1}^{k-1}\bigcup_{j=1}^4
\big(h^{-2}[a^j_i],h^2[2a^j_i]\big)
\eea
covers $(h^{-2}[a^1_1],h^2[2a^1_k])\subset\R^+$ for any $k\in\N$.

Now we obtain the following estimate
\bean
4\dim\FH^{(h^{-1}[a^1_1],h[2a^1_k])}(\A^f_\theta)
&\geq\sum_{i=1}^{k-1}\sum_{j=1}^4
\dim\FH^{(h^{-1}[a^j_i],h[2a^j_i])}(\A^f_\theta)
+\dim\FH^{(h^{-1}[a^1_k],h[2a^1_k])}(\A^f_\theta) \\
&\geq\sum_{i=1}^{k-1}\sum_{j=1}^4
\dim\FH^{(h^{-2}[a^j_i],h^2[2a^j_i])}(\A^f)
+\dim\FH^{(h^{-2}[a^1_k],h^2[2a^1_k])}(\A^f) \\
&\geq\dim\FH^{(h^{-2}[a^1_1],h^2[2a^1_k])}(\A^f).
\eea
This proves the lemma. 

\end{proof}

\begin{Prop}\label{lem:dimcompa}
Fix $F\in\D(\Sigma)$ and $f\in\widetilde\F(\frac{1}{6})$,
see Definition \ref{def:F(a,r)}.
Let $\overline c, \overline c\,',\overline\epsilon>0$ be the constants from Lemma \ref{lem:m(s)f}.
Choose a generic pair $(H,\theta)$.
If $\max\{h^{-1}[2],h^{-1}[2\overline c(H,\theta)]\}<a<T<\infty$
then there exists
\bean
n=n(N,g,F,\overline c\,',\overline\epsilon,H,\theta)\in\N
\eea
such that
\bean
\dim\FH^{(a,T)}(\A^f_\theta)\geq
\frac{1}{4^n}\dim\FH^{(h^{-n}[a],h^{n}[T])}(\A^f)
\eea
holds for generic $a,T$, see (\ref{eqn:defofh[p]}) for $h[p]$.
\end{Prop}

\begin{proof}
In order to use Lemma \ref{lem:dimcompa0},
we first introduce a sequence of primitive of magnetic perturbation $\{\theta^i\}_{i=0}^n\subset\mathcal P$  
which satisfies the following properties:
\begin{itemize}
\item $\theta^i=d^i\theta$, where $0=d^0<d^1<\cdots<d^n=1$;
\item $\A^f_{\theta^i}:\L\times\R\to\R$ is Morse for all $i=0,1,\dots,n$;
\item $C^i=(d^{i+1}-d^i)\|\theta\|_\infty$ satisfies 
the assumption of Proposition \ref{prop:f_etaesm(s)} and Lemma \ref{lem:acvales} for all $i=0,1,\dots,n-1$.
\end{itemize}
Note that the {\em subdivision number} $n$ for $\theta$ does not depend on the action window.
Now choose $a,T$ such that the following conditions hold:
\begin{itemize}
\item $\max\{h^{-1}[2],h^{-1}[2\overline c(H,\theta)]\}<a<T<\infty$; 
\item $h^{-n+i}[a],h^{n-i}[T]\notin A(\A^f_{\theta^i})\ \forall i=0,1,\dots,n$.
\end{itemize}
By the above second condition,
$\FH^{(h^{-n+i}[a],h^{n-i}[T])}(\A^f_{\theta^i})$
are well-defined for $0\leq i\leq n$.
Now we are ready to apply Lemma \ref{lem:dimcompa0}.
If we use (\ref{eqn:compa0con}) inductively then we conclude that
\bean
\dim\FH^{(a,T)}(\A^f_\theta)
&=\dim\FH^{(a,T)}(\A^f_{\theta^n}) \\
&\geq\frac{1}{4}\dim\FH^{(h^{-1}[a],h[T])}(\A^f_{\theta^{n-1}}) \\
&\geq \cdots \\
&\geq\frac{1}{4^n}\dim\FH^{(h^{-n}[a],h^{n}[T])}(\A^f_{\theta^0}) \\
&=\frac{1}{4^n}\dim\FH^{(h^{-n}[a],h^{n}[T])}(\A^f_\theta).
\eea
This proves the lemma. 
\end{proof}

\begin{Rmk}\label{rmk:ncondi}
The argument in Proposition \ref{lem:dimcompa} holds for any $F\in\D(\Sigma)$ and any generic $(H,\theta)$.
Note $\overline c\,',\overline\epsilon$ depend on $F$ and a $\delta-$\nbd of $F^{-1}(0)$. 
For a given diffeomorphism $\varphi$,
consider all defining data $(H,\theta)$ for $\varphi$ 
such that $\varphi=\phi_{X_H^\sigma}^1$.
Now we consider
\bean
n':=\inf_{(H,\theta)}\inf_{(F,\delta)}n(N,g,F,\overline c\,',\overline\epsilon,H,\theta)
\eea
then $n'$ depends only on $(N,g,\Sigma,\varphi)$.
By abuse of notation, we write $n=n'$. 
\end{Rmk}

\begin{proof}[Proof of Theorem \ref{thm:fconti}]
We first fix a defining Hamiltonian $F$ for $\Sigma$
and a defining data $(H,\theta)$ for $\varphi$.
Choose $f\in\widetilde\F(\frac{1}{6})$,
see Definition \ref{def:F(a,r)}.
By the generic assumption,
$\varphi$ has no periodic leaf-wise intersection point and
$\A^f_\theta$ is Morse for the action window 
$\big(\frac{1}{6}+\overline c(H,\theta),\infty\big]$, see Corollary \ref{cor:afmorse}.

If we choose a generic action value $a,T$ such that 
$\max\{h^{-1}[2],h^{-1}[2\overline c(H,\theta)]\}<a<T<\infty$
then Proposition \ref{prop:critmlwip} and Proposition \ref{lem:dimcompa} imply that
\bea\label{eqn:mlwipesdim}
n(T)
&\geq\#\Crit^{(\frac{1}{6}+\overline c(H,\theta),T-\overline c(H,\theta))}(\A^f_\theta) \\
&\geq\dim\FH^{(a,T-\overline c(H,\theta))}(\A^f_\theta) \\
&\geq\frac{1}{4^n}\dim\FH^{(h^{-n}[a],h^n[T-\overline c(H,\theta)])}(\A^f).
\eea
Here $n=n(N,g,\Sigma,\varphi)\in\N$ is the constant from 
Proposition \ref{lem:dimcompa} with Remark \ref{rmk:ncondi}.

Now we recall that $\L_N$ is the free loop space of $(N,g)$. 
The energy functional $\E:\L_N\to\R$ is given by
\bean
\E(q):=\int_0^1\frac{1}{2}|\dot q|^2dt.
\eea
For given $0<T<\infty$, denote by
\bean
\L_N(T):=\bigg\{q\in\L_N\,:\,\E(q)\leq\frac{1}{2}T^2\bigg\}.
\eea
By the result of Macarini-Merry-Paternain \cite[Proof of Theorem A, Remark 1.4]{MMP},
there exists a constant $c'=c'(N,g,\Sigma,\varphi)>0$ such that
\bean
\dim\FH^{(h^{-n}[a],h^n[T-\overline c(H,\theta)])}(\A^f)\geq
\rank\{\iota:\H_*\big(\L_N(c'(T-1))\big)\to\H_*(\L_N)\}.
\eea
If $c:=\max\{4^n,c'\}>0$ then finally we obtain
\bean
n(T)
\geq&\frac{1}{4^n}\rank\{\iota:\H_*\big(\L_N(c'(T-1))\big)\to\H_*(\L_N)\} \\
\geq&\frac{1}{c}\cdot\rank\{\iota:\H_*\big(\L_N(c(T-1))\big)\to\H_*(\L_N)\}.
\eea
This proves the theorem.
\end{proof}

\appendix

\section{The perturbed Rabinowitz action functional is \\ generically Morse.}\label{app:gen}

In this section we study the Morse property of
the perturbed Rabinowitz action functional.
Note first that the action functional $\A_\theta=\A^F_{H,\theta}$
is determined by the following data $F\in\D(\Sigma)$, $H\in\Hc$ and
$\theta\in\mathcal P$.
We claim that $\A_\theta$ is Morse for generic $(H,\theta)\in\Hc\times\mathcal P$.
The generic property for $H\in\Hc$ is well-studied in \cite[Appendix A]{AF09}.
So we additionally consider the Morse property of $\A_\theta$ 
with respect to $\theta\in\mathcal P$.
First recall that
\bean
\mathcal P=\{\theta\in C^\infty(\R/\Z,\Om^1(\widetilde N))\,:\,\theta_t=0,\ \forall t\in[0,\frac{1}{2}]\text{ and }
\theta_t\text{ is bounded},\ \forall t\in[\frac{1}{2},1]\}.
\eea

\begin{Thm}\label{thm:gen}
For a generic pair $(H,\theta)\in\Hc\times\mathcal P$,
the perturbed Rabinowitz action functional $\A_\theta$ is Morse.
\end{Thm}

\subsection{Preparations}
In order to prove the genericity of the Morse property,
we follow the standard method.
Let us consider a certain linear operator and show its surjectivity
then Theorem \ref{thm:gen} deduced from Sard-Smale's theorem.
In this proof we follow the strategy of \cite[Appendix A]{AF09}.

First, let us recall the definition of the perturbed Rabinowitz action functional
\bean
&\A_\theta:\L\times\R\to\R\\
&\A_\theta(u,\eta)=\int_0^1u^*\lambda-\eta\int_0^1F(t,u(t))dt-\int_0^1H(t,u(t))dt
+\int_0^1 \tau^*\theta_t(\widetilde u(t))[\frac{d}{dt}\widetilde u(t)]dt.
\eea
Here, in this section, $\L\equiv W^{1,2}(\R/\Z,T^*N)$ is the completed loop space of $T^*N$.
For notational convenience we adopt the functionals $\F:\L\to\R$ and $\A^{\eta_0}_\theta:\L\to\R$
defined by
\bean
F(u):=\int_0^1F(t,u)dt,
\qquad
\A^{\eta_0}_\theta(u):=\A_\theta(u,\eta_0)
\eea 
for a fixed $\eta_0\in\R$.
We note that $\A_\theta(u,\eta)=\A^{\eta_0}_\theta(u)+(\eta_0-\eta)\F(u)$, and we obtain
\bean
d\A_\theta(u,\eta)[\hat u,\hat\eta]=d\A^{\eta_0}_\theta(u)[\hat u]
-\hat\eta\F(u)+(\eta_0-\eta)d\F(u)[\hat u]
\eea
where $\hat u\in\Gamma^{1,2}(u^*T(T^*N))$, the space of $W^{1,2}$ vector fields along $u$ and $\eta\in\R$.
For a critical point $w_0=(u_0,\eta_0)\in\Crit(\A_\theta)$ the Hessian at $w_0$ equals
\bean
\He_{\A_\theta}(w_0)[(\hat u_1,\hat \eta_1),(\hat u_2,\hat \eta_2)]
=\He_{\A^{\eta_0}_\theta}(u_0)[\hat u_1,\hat u_2]-\hat\eta_1 d\F(u_0)[\hat u_2]-\hat\eta_2 d\F(u_0)[\hat u_1].
\eea

For a function $(\eta_0F+H):[0,1]\times T^*N\to\R$ and 
an $\R/\Z$-parametrized symplectic form $\om_\sigma$, 
we consider a Hamiltonian type diffeomorphism $\psi$ which is a
time-1-map of $X_{\eta_0F+H}^\sigma$.
We then define
\bea\label{twloop}
\L_{\psi}:=\{v\in W^{1,2}([0,1],T^*N)\,:\,v(0)=\psi(v(1))\},
\eea
the twisted loop space, and introduce the diffeomorphism $\Psi:\L_\psi\to\L$
\bean
\Psi(v)(t)=\psi^t(v(t))
\eea
where $\psi^t$ is a time-$t$-map of $X_{\eta_0F+H}^\sigma$.
For a fixed critical point $w_0=(u_0,\eta_0)$ of $\A_\theta$ we use this diffeomorphism to pull back $\A_\theta$
\bean
\A^\theta:=\A_\theta\circ (\Psi\times\id_\R):\L_\psi\times\R\to\R.
\eea
We set $v_0:=\Psi^{-1}\circ u_0$, thus $v_0=$\ const. Then we simplify the Hessian as follows 
\bean
\He_{\A^\theta}(v_0,\eta_0)[(\hat v_1,\hat\eta_1),(\hat v_2,\hat\eta_2)]
=\int_0^1\om(\frac{d}{dt}\hat v_1,\hat v_2)dt-\hat\eta_1 d\overline\F(v_0)[\hat v_2]-\hat\eta_2 d\overline\F(v_0)[\hat v_1],
\eea
where $\overline\F:=\F\circ\Psi$. 

Recall from Definition \ref{def:defham} that
$F(t,x)=\rho(t)\bar{F}(x)$.
Since $\rho(t)=0$ for $t\in[\frac{1}{2},1]$,
$\psi^t$ preserves the level of $\bar{F}$ for $t\in[\frac{1}{2},1]$ 
and $H(t,x),\sigma_t$ vanish for $t\in[0,\frac{1}{2}]$, 
we compute
\bean
\overline\F(v)&=\int_0^1F(t,\psi^t(v))dt=\int_0^{\frac{1}{2}}F(t,\psi^t(v))dt \\
&=\int_0^{\frac{1}{2}}F(t,v)dt=\int_0^1F(t,v)dt.
\eea  
Thus, the Hessian of $\A^\theta$ becomes
\bea\label{eqn:hess_11}
\He_{\A^\theta}&(v_0,\eta_0)[(\hat v_1,\hat\eta_1),(\hat v_2,\hat\eta_2)] \\
&=\int_0^1\om(\frac{d}{dt}\hat v_1,\hat v_2)dt-\hat\eta_1 \int_0^1dF(t,v_0)[\hat v_2]dt-\hat\eta_2 \int_0^1dF(t,v_0)[\hat v_1]dt
\eea

\subsection{The linearized operator.}
We denote by
\bean 
\Hc^k&:=\{H\in C^k(\R/\Z\times T^*N)\,:\,H(t,\cdot)=0,\ \forall t\in[0,\frac{1}{2}]\}; \\
\mathcal P^k&:=\{\theta\in C^k(\R/\Z,\Om^1(\widetilde N))\,:\,\theta_t=0,\ \forall t\in[0,\frac{1}{2}]\text{ and }
\theta_t\text{ is bounded},\ \forall t\in[\frac{1}{2},1]\}.
\eea
For $v\in\L_\psi$, see (\ref{twloop}), we define the bundle 
$\E_\psi\to\L_\psi$ by
\bean
(\E_\psi)_v:=L^2([0,1],v^*T(T^*N)).
\eea
\begin{Def}
Let $(u_0,\eta_0)$ be a critical point of $\A_\theta$ and $(v_0,\eta_0)$ the corresponding critical point of
$\A^\theta$, that is the constant loop $v_0$ defined by the equation $u_0=\Psi(v_0)$.
Then we define the linear operator
\bean
L_{(v_0,\eta_0,H,\theta)}:
T_{(v_0,\eta_0,H,\theta)}(\L_\psi\times\R\times\Hc^k\times\mathcal P^k)\to 
(\E_\psi)^\vee\times\R
\eea
where $(\E_\psi)^\vee$ is the vertical subspace of the bundle $\E_\psi$.
Then we obtain
\bean
\langle L_{(v_0,\eta_0,H,\theta)}[\hat v_1,\hat\eta_1,\hat H,\hat\theta],(\hat v_2,\hat\eta_2)\rangle\quad&\\
:=\He_{\A^\theta}(v_0,\eta_0)[(\hat v_1,\hat\eta_1),(\hat v_2,\hat\eta_2)]
&+\int_0^1 (\Psi^* d \hat H)(t,v_0)[\hat v_2(t)]dt \\
&+\int_0^1(\Psi^*\tau^*d\hat\theta_t)(\frac{d}{dt}\hat v_1(t),\hat v_2(t))dt.
\eea
\end{Def}

\begin{Prop}\label{prop:linop}
The operator $L_{(v_0,\eta_0,H,\theta)}$ is surjective. Indeed, $L_{(v_0,\eta_0,H,\theta)}$ is surjective 
when restricted to the space
\bean
\V:=\{(\hat v,\hat\eta,\hat H,\hat\theta)\in 
T_{(v_0,\eta_0,H,\theta)}(\L_\psi\times\R\times\Hc^k\times\mathcal P^k)\,:\,\hat v(\frac{1}{2})=0\}.
\eea
\end{Prop}

\begin{proof}
The $L^2$-Hessian is a self-adjoint operator. 
Thus, the operator $L_{(v_0,\eta_0,H,\theta)}$ has closed image.
Therefore, it suffices to prove that the annihilator of the image of $L_{(v_0,\eta_0,H,\theta)}$ is zero.
Let $(\hat v_2,\hat\eta_2)$ be in the annihilator of the image of $L_{(v_0,\eta_0,H,\theta)}$, that means
\bean
\langle L_{(v_0,\eta_0,H,\theta)}[\hat v_1,\hat\eta_1,\hat H,\hat\theta],(\hat v_2,\hat\eta_2)\rangle=0
\eea
for all $(\hat v_1,\hat\eta_1,\hat H,\hat\theta)\in T_{(v_0,\eta_0,H,\theta)}(\L_\psi\times\R\times\Hc^k\times\mathcal P^k)$.
This is equivalent to the following three equations:
\bea\label{eqn:hess_v,eta}
\He_{\A^\theta}(v_0,\eta_0)[(\hat v_1,\hat\eta_1),(\hat v_2,\hat\eta_2)]=0,
\quad\forall(\hat v_1,\hat\eta_1)\in(T_{v_0}\L_\psi)\times\R;
\eea
\bea\label{eqn:hess_H}
\int_0^1 d\hat H_t(\psi^t(v_0))[d\psi^t(v_0)[\hat v_2]]=0,
\quad\forall\hat H\in\Hc^k ;
\eea
\bea\label{eqn:hess_beta}
\int_0^1d\hat\theta_t(\psi^t(v_0))\big[\tau_*d\psi^t(v_0)[\frac{d}{dt}\hat v_1(t)],\tau_*d\psi^t(v_0)[\hat v_2(t)]\big]dt=0,
\quad\forall(\hat v_1,\hat\theta)\in T_{v_0}\L_\psi\times\mathcal P^k.
\eea
Note that the Hessian $\He_{\A^\theta}$ is a self-adjoint operator, 
equations (\ref{eqn:hess_11}) and (\ref{eqn:hess_v,eta}) with elliptic regularity implies that
$\hat v_2\in C^{k+1}([0,1],T_{v_0}T^*N)$ and $(\hat v_2,\hat\eta_2)$ satisfies the equation
\bea\label{eqn:v_2condi}
\frac{d}{dt}\hat v_2-\hat\eta_2 X_F(t,v_0)=0
\eea
and the linearized boundary condition
\bea\label{eqn:v_2bound}
\hat v_2(0)=d\psi(v_0)[\hat v_2(1)].
\eea

Equation (\ref{eqn:hess_H}) implies that
\bea\label{eqn:v_2van}
\hat v_2(t)=0,\quad\forall t\in[\frac{1}{2},1].
\eea
Recall from (\ref{eqn:hamvec}) that $X_F(t,x)=\rho(t)X_{\bar{F}}(x)$ then (\ref{eqn:v_2condi}) becomes
\bean
\frac{d}{dt}\hat v_2-\hat\eta_2\rho(t)X_{\bar{F}}(v_0)=0.
\eea
This is a linear ODE in the vector space $T_{v_0}T^*N$ as follows
\bea\label{eqn:v_2rho}
\hat v_2(t)=\hat v_2(0)+\hat\eta_2\left(\int_0^t\rho(\tau)d\tau\right)X_{\bar{F}}(v_0).
\eea
Recall from (\ref{eqn:rhocondi}) that $\int_0^t\rho(\tau)d\tau=1$ for all $t\in[\frac{1}{2},1]$.
Substitute this into (\ref{eqn:v_2rho}) and combine with equation (\ref{eqn:v_2van}), we then obtain 
\bea\label{eqn:v_2>1/2}
0=\hat v_2(t)=\hat v_2(0)+\hat\eta_2X_{\bar{F}}(v_0)
\eea
for $t\geq\frac{1}{2}$.
By using equations (\ref{eqn:v_2bound}) and (\ref{eqn:v_2van}) at $t=1$, we deduce $\hat v_2(0)=0$.
Now, put this into (\ref{eqn:v_2>1/2}) we have
\bean
\hat\eta_2X_{\bar{F}}(v_0)=0
\eea
Since $(v_0,\eta_0)$ is deduced from a critical point $(u_0,\eta_0)$ of $\A_\theta$,
we have $\bar{F}(v_0)=\bar{F}(u(0))=k$, and we already assume that $k$ is a regular value of $\bar{F}$.
In particular,
\bea\label{eqn:eta_2condi}
\hat\eta_2=0
\eea
Equation (\ref{eqn:v_2>1/2}), (\ref{eqn:eta_2condi}) and $\hat v_2(0)=0$ imply
\bean
\hat v_2(t)=0,\quad\forall t\in[0,1].
\eea
Therefore, the annihilator of the image of $L_{(v_0,\eta_0,H,\theta)}$ vanishes 
and thus $L_{(v_0,\eta_0,H,\theta)}$ is surjective.
Moreover, if we restrict the domain of $L_{(v_0,\eta_0,H,\theta)}$ to $\V$
then the only change occurs in (\ref{eqn:v_2condi}) at $t=\frac{1}{2}$.
By continuity, however, equation (\ref{eqn:v_2condi}) is still valid for all $t\in[0,1]$.
\end{proof}

\begin{Rmk}
In the proof of Proposition \ref{prop:linop}, we do not use equation (\ref{eqn:hess_beta}).
This means that for a fixed $\theta\in\mathcal P^k$ there exists $H\in\Hc^k$ such that 
$L_{(v_0,\eta_0,H,\theta)}$ is surjective.
\end{Rmk}

\begin{proof}[Proof of Theorem \ref{thm:gen}]
We first define the Banach space bundle $\E\to\L$ by 
\bean
\E_u=L^2(\R/\Z,u^*T(T^*N))
\eea 
for $u\in\L$. 
Now consider the section $S:\L\times\R\times\Hc^k\times\mathcal P^k\to\E^\vee\times\R$ 
given by the differential of the Rabinowitz action
functional $\A_\theta$
\bea\label{eqn:defS}
S(u,\eta,H,\theta):=d\A_\theta(u,\eta).
\eea
Here $(H,\theta)\in\Hc^k\times\mathcal P^k$ is the additional variables for the perturbation of $\A_\theta$.
Its vertical differential $DS:T_{(u_0,\eta_0,H,\theta)}(\L\times\R\times\Hc^k\times\mathcal P^k)
\to T_{S(u_0,\eta_0,H,\theta)}(\E^\vee\times\R)$ 
at $(u_0,\eta_0,H,\theta)\in S^{-1}(0)$ is 
\bea
\langle D&S_{(u_0,\eta_0,H,\theta)}[(\hat u_1,\hat\eta_1,\hat H,\hat\theta)],(\hat u_2,\hat\eta_2)\rangle \\
&=\He_{\A_\theta}(u_0,\eta_0)[(\hat u_1,\hat\eta_1),(\hat u_2,\hat\eta_2) ]
+\int_0^1d\hat H(t,u_0)[\hat u_2(t)]dt
+\int_0^1d\hat\theta_t(\tau_*\frac{d}{dt}\hat u_1(t),\tau_*\hat u_2(t))dt
\eea
Since $(\Psi\times\id_\R\times\id_{\Hc^k}\times\id_{\mathcal P^k})^*DS=L_{(v_0,\eta_0,H,\theta)}$, 
the operator $DS$ is surjective.
Thus, by the implicit function theorem the moduli space 
\bean
\M:=S^{-1}(0)
\eea
is a smooth Banach manifold. We consider the projection $\Pi_{\Hc^k\times\mathcal P^k}:\M\to\Hc^k\times\mathcal P^k$. 
Then the $\A_\theta$ is Morse if and only if $(H,\theta)$ is a regular value of $\Pi_{\Hc^k\times\mathcal P^k}$.
By the Sard-Smale theorem this forms a generic set for $k$ large enough.
Moreover, the Morse condition is $C^k$-open. 
Thus for function in an open and dense subset of $\Hc^k\times\mathcal P^k$, 
the Rabinowitz action functional is Morse. 
Taking the intersection of all $k$ concludes the proof of Theorem \ref{thm:gen}.
\end{proof}

Now we discuss the Morse property of $\A^f_\theta$.
Since we are interested in critical points of $\A^f_\theta$ with positive action value,
it suffices to check the Morse property for the positive critical points.

\begin{Cor}\label{cor:afmorse}
Given $a>0$ and choose $f\in\F(a)$ (see, Definition \ref{def:F(a)}).
For a generic $(H,\theta)\in\Hc\times\mathcal P$ the perturbed $\F$-Rabinowitz action functional $\A^f_\theta=\A^{F,f}_{H,\theta}$ 
is Morse on the action window $(a+c(H,\theta),\infty]$.
\end{Cor}

\begin{proof}
Let $w_0=(u_0,\eta_0)$ be a critical point of $\A^f_\theta$ with $\A^f_\theta(u_0,\eta_0)>a-c(H,\theta)$ 
then by the argument in Proposition \ref{prop:critmlwip} we obtain $f(\eta_0)>a$.
Since $f\in\F(a)$, see Definition \ref{def:F(a)}, we conclude $f'(\eta_0)=1$. 
Hence the argument in the proof of Theorem \ref{thm:gen} definitely holds.
This proves the corollary.
\end{proof}

\section{No periodic magnetic leaf-wise intersection points}
\label{sec:noper}
In this section, we study the second regularity property of 
$\varphi$, see Definition \ref{def:Mreg}.
The claim is that $\varphi$ has no periodic leaf-wise intersection points 
for generic $H\in\Hc$ and $\theta\in\mathcal P$.
In \cite{AF09b} Albers-Frauenfelder already studied the above
property with respect to $H\in\Hc$.
As in Appendix \ref{app:gen}, we work with $(H,\theta)\in\Hc\times\mathcal P$
and modify the strategy of \cite{AF09b}.

Recall that the hypersurface $\Sigma\subset T^*N$ is called non-degenerate 
if closed Reeb orbits on $\Sigma$ form a discrete set.
A generic $\Sigma$ is non-degenerate, see \cite[Theorem B.1]{CF09}.
If the critical points of $\A_\theta$ does not meet any closed Reeb orbit
then there are no periodic leaf-wise intersection points.
Thus it suffices to prove the following theorem.

\begin{Thm}\label{thm:infgen}
Let $\Sigma\subset T^*N$ be a non-degenerate starshaped hypersurface and 
$\Rb$ be a set of closed Reeb orbit on $\Sigma$ which form a discrete set.
If $\dim N\geq 2$ then the set
\bea\label{eqn:Hbetasigma}
\{(H,\theta)\in\Hc\times\mathcal P\,:\,\A_\theta \text{ is Morse and }\im(x)\cap\im(y)=\emptyset,\ \forall 
x\in\Crit(\A_\theta),\ y\in\Rb\}
\eea
is generic in $\Hc\times\mathcal P$, see Definition \ref{def:betadef} and \ref{def:perham}.
\end{Thm}
\begin{proof}
We first define the evaluation map $\ev:\M\to\Sigma$
\bean
\ev(u_0,\eta_0,H,\theta)= u_0(\frac{1}{2})
\eea
where $\M$ is the same as in the proof of Theorem \ref{thm:gen}.
Proposition \ref{prop:linop} with Lemma \ref{lem:surj} below guarantee that the evaluation map
\bean
\ev_{(H,\theta)}:=\ev(\cdot,\cdot,H,\theta):\Crit(\A_\theta)\to\Sigma
\eea 
is a submersion for a generic choice of $(H,\theta)$.
Let $\Rb^n$ be the set of Reeb orbit with period less than $n$
which is a 1-dimensional set in $\Sigma$,
then $\ev_{(H,\theta)}^{-1}(\Rb^n)$ does not intersect $\Crit(\A_\theta)$
since $\dim \Sigma\geq3$. 
Therefore, the set
\bea\label{eqn:HbetaSigman}
\{(H,\theta)\in\Hc\times\mathcal P\,:\,\A_\theta \text{ is Morse and }\im(x)\cap\im(y)=\emptyset,\ \forall 
x\in\Crit(\A_\theta),\ y\in\Rb^n\}
\eea
is generic in $\Hc\times\mathcal P$ for all $n\in\N$. 
Now, the set (\ref{eqn:Hbetasigma}) is a countable intersection 
of the set (\ref{eqn:HbetaSigman}), for all $n\in\N$.
This proves the Theorem \ref{thm:infgen}
\end{proof}

The following lemma is contained in \cite{AF09b}.

\begin{Lemma}\label{lem:surj}
Let $\E\to\S$ be a Banach bundle and $s:\S\to\E$ a smooth section.
Moreover, let $\Phi:\S\to \mathcal{C}$ be a smooth map into the Banach manifold $\mathcal{C}$.
We fix a point $x\in s^{-1}(0)\subset\S$ and set $K:=\ker d\Phi(x)\subset T_x\S$
and assume the following two conditions.
\begin{enumerate}
\item The vertical differential $Ds|_K:K\to\E_x$ is surjective.
\item $d\Phi(x):T_x\S\to T_{\Phi(x)}\mathcal{C}$ is surjective.
\end{enumerate}
Then $d\Phi(x)|_{\ker Ds(x)}:\ker Ds(x)\to T_{\Phi(x)}\mathcal{C}$ is surjective.
\end{Lemma}

\end{document}